\documentclass[11pt,letterpaper]{amsart}
\usepackage[margin=1.08in]{geometry}
\usepackage[T1]{fontenc}
\usepackage[utf8]{inputenc}
\usepackage{amsmath,amssymb,amsthm,mathtools}
\usepackage{microtype}
\IfFileExists{libertinus.sty}{\usepackage{libertinus}}{\usepackage{lmodern}}
\usepackage{enumitem}
\usepackage{xcolor}
\DeclareRobustCommand{\rededit}[1]{{\color{black}#1}}
\usepackage{cite}
\usepackage{hyperref}
\usepackage{aliascnt}
\usepackage[nameinlink,capitalize]{cleveref}

\hypersetup{
  colorlinks=true,
  linkcolor=black,
  citecolor=black,
  urlcolor=black,
  pdftitle={Polynomial growth of complex polynomial Bohnenblust--Hille constants},
  pdfauthor={Daniel M. Pellegrino and Eduardo V. Teixeira}
}
\setlist{itemsep=0.25em,topsep=0.45em}
\numberwithin{equation}{section}
\allowdisplaybreaks[1]
\newcommand{\C}{\mathbb C}
\newcommand{\R}{\mathbb R}
\newcommand{\N}{\mathbb N}
\newcommand{\D}{\mathbb D}
\newcommand{\T}{\mathbb T}
\newcommand{\e}{\mathrm e}
\newcommand{\W}{\mathcal W}

\newcommand{\Sid}{S}
\DeclareMathOperator{\card}{card}

\newtheorem{theorem}{Theorem}[section]
\newaliascnt{proposition}{theorem}
\newtheorem{proposition}[proposition]{Proposition}
\aliascntresetthe{proposition}
\newaliascnt{lemma}{theorem}
\newtheorem{lemma}[lemma]{Lemma}
\aliascntresetthe{lemma}
\newaliascnt{corollary}{theorem}
\newtheorem{corollary}[corollary]{Corollary}
\aliascntresetthe{corollary}
\theoremstyle{remark}
\newtheorem{remark}[theorem]{Remark}
\theoremstyle{plain}
\newtheorem{theoremA}{Theorem}
\newtheorem{theoremB}{Theorem}
\newtheorem{theoremC}{Theorem}
\newtheorem{theoremD}{Theorem}

\crefname{theorem}{theorem}{theorems}
\Crefname{theorem}{Theorem}{Theorems}
\crefname{proposition}{proposition}{propositions}
\Crefname{proposition}{Proposition}{Propositions}
\crefname{lemma}{lemma}{lemmas}
\Crefname{lemma}{Lemma}{Lemmas}
\crefname{corollary}{corollary}{corollaries}
\Crefname{corollary}{Corollary}{Corollaries}

\title[Polynomial growth of complex polynomial BH constants]
{Polynomial Growth of Complex Polynomial Bohnenblust--Hille Constants}

\author[D. M. Pellegrino]{Daniel M. Pellegrino}
\address{Departamento de Matem\'atica, Universidade Federal da Para\'iba, Jo\~ao Pessoa, PB, Brazil}
\email{daniel.pellegrino@academico.ufpb.br}

\author[E. V. Teixeira]{Eduardo V. Teixeira}
\address{Department of Mathematics, Oklahoma State University, Stillwater, OK 74078, USA}
\email{eduardo.teixeira@okstate.edu}

\date{}

\begin{document}

\begin{abstract}
For complex $m$-homogeneous polynomials, let $D_m$ denote the optimal
dimension-free constant in the polynomial Bohnenblust--Hille inequality.
We prove that these constants have at most polynomial growth:
$D_m\le K m^{B_0}$ for absolute constants $B_0,K>0$.
This replaces the earlier bound $\exp(O(\sqrt{m\log m}))$ by a fixed
power of the degree. The proof combines coefficient-preserving degree
reduction with a weighted bootstrap: balanced splits contract through
entropy, while dominant powers are compressed to lower degree.

\rededit{We also identify the exact asymptotic constant on the critical diagonal:}
$D_{m,n_m}\to2$ whenever $n_m/m\to1$. \rededit{Thus the value $2$ is attained
asymptotically in the critical linear-dimensional regime and, in particular,}
$\liminf_mD_m\ge2$. Applications give polynomial degree dependence in
homogeneous Sidon estimates and a logarithmic quantitative remainder in
the multidimensional Bohr-radius asymptotic.
\end{abstract}
\vspace{-1cm}
\maketitle
\tableofcontents

\section{Introduction}\label{sec:introduction}

The Bohnenblust--Hille inequality belongs to the classical theory of Dirichlet
series.  In 1913 Bohr showed that the gap between uniform and absolute
convergence of an ordinary Dirichlet series can be at most $1/2$.  The problem
remained central until 1931, when Bohnenblust and Hille proved that the value
$1/2$ is sharp \cite{BohnenblustHille1931}; see also the historical discussion
in \cite{DefantEtAl2011}.  A crucial ingredient in their argument is the
coefficient inequality which, in polynomial form, says that for every
$m$-homogeneous polynomial
\[
 P(z)=\sum_{|\alpha|=m}a_\alpha z^\alpha
\]
on the complex polydisc,
\begin{equation}\label{eq:intro-BH-historical}
 \left(\sum_{|\alpha|=m}|a_\alpha|^{2m/(m+1)}\right)^{(m+1)/(2m)}
 \le D_m\|P\|_\infty,
\end{equation}
where
$\|P\|_\infty:=\sup_{z\in\mathbb D^n}|P(z)|$.  For fixed $m,n\ge1$, let
$D_{m,n}$ denote the optimal constant in \eqref{eq:intro-BH-historical} for
$m$-homogeneous polynomials on $\mathbb C^n$.  Thus
\[
 D_m=\sup_{n\ge1}D_{m,n}.
\]
The exponent
$2m/(m+1)$ is optimal.  Thus a problem about Dirichlet series produced a
family of dimension-free inequalities for homogeneous polynomials whose
constants carry substantial quantitative information.

For a long time that quantitative information was very costly.  The classical
passage from multilinear forms to polynomials uses polarization, and the
resulting polynomial bounds are superexponential in the degree; the estimate
recalled in \cite{DefantEtAl2011}, for example, has size $m^{m/2}$ up to
exponential factors.  The work \cite{DefantEtAl2011} changed this
picture by proving hypercontractivity, namely $D_m\le C^m$ for an absolute
constant $C>1$.  That improvement already fed directly into estimates for
Sidon constants and the multidimensional Bohr radius.  The paper \cite{BayartPellegrinoSeoane2014} then pushed
the constants below every exponential scale, proving
\begin{equation}\label{eq:intro-previous-bound}
 D_m\le \exp\!\bigl(O(\sqrt{m\log m})\bigr),
\end{equation}
and used that improvement to obtain the first-order asymptotic
$\mathfrak b_n\sim\sqrt{(\log n)/n}$, where $\mathfrak b_n$ denotes the Bohr radius of the polydisc $\mathbb D^n$.

The question is whether a superpolynomial loss is necessary at all.
Our main theorem shows that it is not: a fixed power of the degree
suffices.

\begin{theoremA}\label{thm:A}
\begingroup
\color{black}
\hypersetup{linkcolor=black,citecolor=black,urlcolor=black}
There are absolute constants $B_0,K>0$ such that
\[
 D_m\le K m^{B_0}\qquad(m\ge1).
\]
\endgroup
\end{theoremA}

Theorem~\ref{thm:A} changes the known scale of the general upper bound,
from subexponential to polynomial. Its proof uses a single unoptimized
choice of parameters, for which $B_0=5$ suffices. A complementary question is what
happens when the dimension varies with the degree. {A particularly natural linear regime is the diagonal $n/m\to1$.} There the elementary monomial count already singles out the number $2$: indeed,
\[
 \binom{m+n-1}{m}^{1/(2m)}\longrightarrow2
 \qquad\text{whenever }\frac nm\longrightarrow1.
\]
\rededit{The next result shows that this scale is exact and identifies the asymptotic
constant itself.}

\begin{theoremB}\label{thm:B}
Let $(n_m)_{m\ge1}$ be a sequence of positive integers such that
\[
 \frac{n_m}{m}\longrightarrow1.
\]
Then
\[
 D_{m,n_m}\longrightarrow2.
\]
Consequently,
\[
 \liminf_{m\to\infty}D_m\ge2.
\]
\end{theoremB}

\rededit{Thus $2$ is the exact asymptotic constant in the critical diagonal regime
$n/m\to1$. In particular, it identifies a concrete asymptotic lower barrier
for the dimension-free constants through $\liminf_mD_m\ge2$.}  The proof of
Theorem~\ref{thm:B} is based on a finite Blaschke--Fej\'er construction that
produces $m$-homogeneous polynomials with supremum norm at most $1$ and
coefficient $\ell_{2m/(m+1)}$-norm tending to $2$.

We next record two consequences in problems historically tied to quantitative
Bohnenblust--Hille estimates.  For $m,n\ge1$ and
\[
 P(z)=\sum_{|\alpha|=m}a_\alpha z^\alpha
 \qquad(z\in\C^n),
\]
let $\Sid_{m,n}$ denote the least constant such that
\[
 \sum_{|\alpha|=m}|a_\alpha|\le \Sid_{m,n}\|P\|_\infty
\]
for every $m$-homogeneous polynomial $P$ on $\C^n$.  The work \cite{DefantEtAl2011} identified, in the range
$n>m^2$, the classical dimension scale up to factors exponential in the
degree.  Theorem~\ref{thm:A} reduces the degree dependence in the upper bound
to a fixed power.  The lower bound below has a different origin: it is obtained from a
unitary-chain construction combined with exact quadratic and quartic averages.

\begin{theoremC}\label{thm:C}
With $B_0,K$ as in Theorem~\ref{thm:A}, for all
$m\ge2$ and $n\ge m^2$, one has
\begin{equation}\label{eq:theoremC-sidon}
 \frac{1}{4\sqrt{\e m}}
 \sqrt{\frac{n^{m-1}}{(m-1)!}}
 \le
 \Sid_{m,n}
 \le
 \e^{1/4}Km^{B_0}
 \sqrt{\frac{n^{m-1}}{(m-1)!}}.
\end{equation}
\end{theoremC}

We finally return to the multidimensional Bohr radius.  For the polydisc
$\mathbb D^n$, the problem is to determine the largest radius for which the
absolute sum of the Taylor coefficients remains controlled by the supremum
norm.  The hypercontractive Bohnenblust--Hille estimate of
\cite{DefantEtAl2011} already gave the correct scale
$\sqrt{(\log n)/n}$ up to absolute multiplicative constants.  The
subexponential bound of \cite{BayartPellegrinoSeoane2014} sharpened this to the
first-order asymptotic
\[
 \mathfrak b_n\sim\sqrt{\frac{\log n}{n}}.
\]
Polynomial growth in Theorem~\ref{thm:A} yields a quantitative
remainder at this scale.

\begin{theoremD}\label{thm:D}
The multidimensional Bohr radius admits a quantitative remainder in its
first-order asymptotic. More precisely,
\[
 \mathfrak b_n=
 \sqrt{\frac{\log n}{n}}
 \left(1+O\!\left(\frac{\log\log n}{\log n}\right)\right).
\]
\end{theoremD}

The proof of Theorem~\ref{thm:D} also yields explicit one-sided constants for
the normalized remainder; these are recorded at the end of its proof.

\medskip
The passage to polynomial growth rests on two features of degree
reduction: diffuse and dominant profiles admit different reductions, and
coefficient ancestry allows their mass to be recovered without cancellation.
A weighted estimate across homogeneous degrees prevents the reduction
losses from accumulating. The Sidon lower bound uses a separate
unitary-chain construction. We describe these ingredients in turn:
\begin{itemize}[leftmargin=*,itemsep=0.55em,topsep=0.45em]
\item \textbf{Diffuse and dominant exponent profiles.}
Fix an $r$-homogeneous layer
\[
 F_r(z)=\sum_{|\alpha|=r}a_\alpha z^\alpha.
\]
The exponent profile of a monomial
\[
 z^\alpha=z_1^{\alpha_1}\cdots z_n^{\alpha_n},
 \qquad |\alpha|=r,
\]
records how its total degree $r$ is distributed among the coordinates.  Fix
$\mu\in(1/2,1)$.  We call the profile \emph{diffuse} if
\[
 \max_j\alpha_j\le\mu r,
\]
that is, no single variable carries more than the fixed fraction $\mu$ of the
total degree.  For example, when $r=12$ and $\mu=0.6$, the monomial
$z_1^4z_2^4z_3^4$ has a diffuse profile.  We call the profile \emph{dominant} if
\[
 \max_j\alpha_j>\mu r.
\]
Since $\mu>1/2$, there can be only one such coordinate: a single variable
carries most of the degree.  For the same values $r=12$ and $\mu=0.6$,
$z_1^8z_2^3z_3$ is dominant.

The distinction determines the degree-reduction mechanism.  For diffuse
profiles, whole-coordinate coloring---each coordinate $z_j$ is sent entirely
to $x_j$ or to $y_j$---distributes the degree between two variable blocks and,
with positive probability, produces balanced bidegrees.  For dominant profiles
this may fail completely, since the large exponent can force almost the entire
degree into one block.  We therefore isolate the dominant coordinate by
projections that do not increase the supremum norm and replace its large power
by a fresh variable through an exact compression.  The common feature of the
two reductions---that the original coefficient can still be identified after
the transformation---is the ancestry property described in the next item.

\item \textbf{Coefficient ancestry and recovery.}
A central new feature of the reduction is the following coefficient-ancestry property. By \emph{coefficient recovery} we mean that, after the degree-reduction map is
applied, every retained child coefficient can still be traced back to one and
only one coefficient of the original polynomial.  This injectivity is essential:
if two different parent monomials produced the same child monomial, their
coefficients would be added before absolute values were taken, and the original
coefficient $\ell_q$-mass could no longer be read from the transformed
polynomial.

For a diffuse profile the recovery is immediate.  Whole-coordinate coloring
sends
\[
 a_\alpha z^\alpha\longmapsto a_\alpha x^\beta y^\gamma,
 \qquad \alpha=\beta+\gamma.
\]
For example, if $\alpha=(3,2,1)$ and the first and third coordinates are colored
$x$ while the second is colored $y$, then
\[
 a_\alpha z_1^3z_2^2z_3
 \longmapsto
 a_\alpha x_1^3y_2^2x_3.
\]
The two child exponent vectors are $\beta=(3,0,1)$ and
$\gamma=(0,2,0)$, and their sum recovers the parent profile
$\alpha=(3,2,1)$ exactly.

For a dominant profile, write
$\alpha=(r-k)e_j+\eta$, where $e_j$ is the $j$th coordinate vector,
$\eta_j=0$, $|\eta|=k$, and $r-k>r/2$.  Contractive projections first
isolate the unique large power, and exact compression then replaces it by a
fresh variable:
\[
 a_\alpha z_j^{r-k}z^\eta\longmapsto a_\alpha x_jz^\eta.
\]
For instance,
\[
 a_\alpha z_1^8z_2^3z_3
 \longmapsto
 a_\alpha x_1z_2^3z_3.
\]
The fresh variable $x_1$ records which coordinate was dominant, while the
residual monomial records all the remaining exponents; the degree layer records
the removed exponent $r-k$.  Hence the original multi-index is again uniquely
recoverable.  In both regimes the reduction changes the degree without losing
the identity of the coefficients, which is precisely what allows the later
weighted estimates to be applied coefficient by coefficient.

\item \textbf{Unitary-chain construction and quadratic/quartic averages.}
The third ingredient has a different purpose.  Here we need a polynomial with
many coefficients of substantial size while keeping its supremum norm under
control.  Let $U_1,\ldots,U_{m-1}$ be independent Haar-distributed unitary
matrices and put $A_r=\sqrt n\,U_r$.  For
$z=(z_1,\ldots,z_n)$, set $D_z:=\operatorname{diag}(z_1,\ldots,z_n)$; then
\[
 P(z)=z^{\mathsf T}A_1D_zA_2D_z\cdots D_zA_{m-1}z,
\]
where $z^{\mathsf T}$ denotes transpose, is $m$-homogeneous.  The unitary
operator norms give a deterministic bound for $\|P\|_\infty$, independent of
any averaging.

The averaging enters in the coefficients.  For a square-free monomial indexed
by an $m$-element set $J$, its coefficient $c_J$ is a sum over the $m!$
orderings of the elements of $J$.  Orthogonality of Haar-unitary matrix entries
makes the quadratic average exact:
\[
 \mathbb E|c_J|^2=m!.
\]
Thus $\sqrt{m!}$ is the natural size suggested by the second moment.  A second
moment alone, however, would still allow the mass to be concentrated in rare,
very large values.  The quartic estimate prevents this: the fourth moment is
controlled by a binary pairing calculation along the unitary chain.  Combining
the exact quadratic average with the quartic bound through interpolation gives
a quantitative lower bound for $\mathbb E|c_J|$.  Summing this lower bound over
the $\binom nm$ square-free monomials produces one realization with large total
coefficient mass, whereas the deterministic unitary bound keeps its supremum
norm small.  This is the mechanism behind the lower Sidon estimate in
Theorem~\ref{thm:C}; Lemma~\ref{lem:binary-pairing-chain} contains the pairing
calculation and Proposition~\ref{prop:sidon-lower} carries out the final
conversion.
\end{itemize}

The surrounding theory includes sharp mixed-exponent multilinear
inequalities~\cite{ABPS}, Bohnenblust--Hille estimates on Boolean
cubes~\cite[Theorem~1]{DMP} and cyclic groups~\cite{SVZ}, and
noncommutative analogues~\cite[Theorem~1.2]{VZ}.
Dimension-free discretization of polynomial suprema~\cite[Theorem~1]{BKSVZ}
and completely bounded inequalities with constant one~\cite[Theorem~1.1]{ADGP}
provide further connections with harmonic analysis and learning theory.
Bohr inequalities for free holomorphic functions on polyballs are studied
in~\cite{Popescu}.

Recent preprints study polynomial-growth Bohnenblust--Hille estimates in
related discrete settings.  Ivanisvili adapts the weighted square-function
and random-coloring mechanism developed here to the Hamming
cube~\cite[Section~1]{Ivanisvili2026}.  Slote and Volberg study the analogous
polynomial-growth problem for products of finite cyclic
groups~\cite{SloteVolberg2026}.  These discrete problems are distinct from
the complex-polydisc setting of the present paper.

\medskip
\noindent\textbf{Notation and conventions.}
Unless explicitly stated otherwise, all polynomials and multilinear forms in
this paper are over $\C$.  Throughout, $\N=\{1,2,\ldots\}$ and
$\N_0=\N\cup\{0\}$.  For $\alpha\in\N_0^n$, write
$|\alpha|=\sum_j\alpha_j$ and $z^\alpha=\prod_j z_j^{\alpha_j}$.  If
$P(z)=\sum_{|\alpha|=m}a_\alpha z^\alpha$, then $a(P)=(a_\alpha)$ and
$|P|_q=\|a(P)\|_q$.  We write $\T=\{z\in\C:|z|=1\}$ and use normalized
Haar measure on its finite products.

\part{Polynomial Bohnenblust--Hille constants}

\section{Classical degree reduction and its barrier}

Let $B_k^{\mathrm{mult}}$ denote the optimal complex $k$-linear
Bohnenblust--Hille constant.  The degree-reduction estimate in
\cite[Theorem~5.2]{BayartPellegrinoSeoane2014} states that, for $1\le k<m$,
\begin{equation}\label{eq:BPS}
 D_m\le C(m,k)B_k^{\mathrm{mult}},
\end{equation}
where
\begin{equation}\label{eq:Cmk}
 C(m,k):=
 \left(1+\frac1k\right)^{(m-k)/2}
 \frac{m^m}{(m-k)^{m-k}}
 \left(\frac{(m-k)!}{m!}\right)^{1/2}.
\end{equation}
The multilinear factor grows only polynomially in $k$; indeed,
\cite[Corollary~3.2]{BayartPellegrinoSeoane2014} gives an absolute constant $\kappa$ such that
\begin{equation}\label{eq:BPS-multilinear-polynomial}
 B_k^{\mathrm{mult}}
 \le
 \kappa k^{(1-\gamma)/2},
\end{equation}
where $\gamma$ is Euler's constant.  It follows that the leading asymptotic
behavior of \eqref{eq:BPS} is governed entirely by the transition factor
$C(m,k)$.

The next estimate identifies the critical scale in the classical
degree-reduction bound \eqref{eq:BPS}; in particular, that bound alone retains a
$\sqrt{m\log m}$ exponential cost.

\begin{proposition}
\label{prop:critical}
Set
\[
 x_m:=\sqrt{\frac{m}{\log m}}.
\]
If $k=k(m)$ satisfies
\[
 \frac{k}{x_m}\longrightarrow t\in(0,\infty),
\]
then
\begin{equation}\label{eq:transition-critical-asymptotic}
 \log C(m,k)
 =
 \frac12\left(t+\frac1t\right)\sqrt{m\log m}
 +o\!\left(\sqrt{m\log m}\right).
\end{equation}
The same asymptotic formula holds with $C(m,k)$ replaced by
$C(m,k)B_k^{\mathrm{mult}}$.  Consequently, among transitions on this
critical scale, the leading coefficient in \eqref{eq:BPS} is minimized
uniquely when
\[
 k=(1+o(1))\sqrt{\frac{m}{\log m}},
\]
and the minimum coefficient is $1$.
\end{proposition}

\begin{proof}
At the stated scale,
\[
 k\longrightarrow\infty,
 \qquad
 k=o(m).
\]
We separate the two contributions to \eqref{eq:Cmk}.  First, Taylor's
formula gives
\begin{align}
 \frac{m-k}{2}\log\left(1+\frac1k\right)
 &=
 \frac{m-k}{2}
 \left(\frac1k+O\!\left(\frac1{k^2}\right)\right) \notag\\
 &=
 \frac{m}{2k}
 +O\!\left(1+\frac{m}{k^2}\right).
 \label{eq:transition-first-cost}
\end{align}

For the remaining factor, Stirling's formula in the form
\[
 \log r!
 =
 \left(r+\frac12\right)\log r-r
 +\frac12\log(2\pi)+O(r^{-1})
\]
yields, uniformly when $k=o(m)$,
\begin{align*}
 &\log\!\left[
 \frac{m^m}{(m-k)^{m-k}}
 \left(\frac{(m-k)!}{m!}\right)^{1/2}
 \right] \\
 &\qquad=
 \frac12\bigl[m\log m-(m-k)\log(m-k)\bigr]
 +\frac{k}{2}
 +O\!\left(1+\frac{k}{m}\right).
\end{align*}
Since
\[
 m\log m-(m-k)\log(m-k)
 =
 k\log m+k+O\!\left(\frac{k^2}{m}\right),
\]
we obtain
\begin{equation}\label{eq:transition-second-cost}
 \log\!\left[
 \frac{m^m}{(m-k)^{m-k}}
 \left(\frac{(m-k)!}{m!}\right)^{1/2}
 \right]
 =
 \frac{k}{2}\log m+k
 +O\!\left(1+\frac{k^2}{m}\right).
\end{equation}

Combining \eqref{eq:transition-first-cost} and
\eqref{eq:transition-second-cost}, we find
\begin{equation}\label{eq:Cmk-expanded}
 \log C(m,k)
 =
 \frac{m}{2k}
 +\frac{k}{2}\log m
 +k
 +O\!\left(
 1+\frac{m}{k^2}+\frac{k^2}{m}
 \right).
\end{equation}
When $k\asymp\sqrt{m/\log m}$, each of
\[
 k,\qquad
 \frac{m}{k^2},\qquad
 \frac{k^2}{m}
\]
is $o(\sqrt{m\log m})$.  Therefore
\begin{equation}\label{eq:Cmk-leading}
 \log C(m,k)
 =
 \frac{m}{2k}
 +\frac{k}{2}\log m
 +o\!\left(\sqrt{m\log m}\right).
\end{equation}

The multilinear factor is negligible at this scale.  Indeed,
\eqref{eq:BPS-multilinear-polynomial} and the elementary lower bound
$B_k^{\mathrm{mult}}\ge1$ give
\[
 0\le \log B_k^{\mathrm{mult}}
 \le
 \log\kappa+\frac{1-\gamma}{2}\log k
 =
 o\!\left(\sqrt{m\log m}\right).
\]
Thus \eqref{eq:Cmk-leading} remains valid with $C(m,k)$ replaced by
$C(m,k)B_k^{\mathrm{mult}}$.

Finally, if
\[
 k=(t+o(1))\sqrt{\frac{m}{\log m}},
\]
then
\[
 \frac{m}{2k}
 =
 \left(\frac{1}{2t}+o(1)\right)\sqrt{m\log m},
 \qquad
 \frac{k}{2}\log m
 =
 \left(\frac{t}{2}+o(1)\right)\sqrt{m\log m}.
\]
Since
\[
 t+\frac1t\ge2,
\]
with equality if and only if $t=1$, the leading coefficient is minimized
uniquely at $t=1$.
\end{proof}

\section{Two-block estimates and a quasipolynomial bound}
\label{sec:quasipolynomial}

\subsection{Two-block estimate}
\label{sec:fractional}

For $r\ge1$, set
\[
 q_r:=\frac{2r}{r+1},
 \qquad
 \frac1{q_r}=\frac12+\frac1{2r}.
\]
For a finite scalar family $a=(a_\nu)$, write
\[
 \|a\|_p
 :=
 \left(\sum_\nu |a_\nu|^p\right)^{1/p}.
\]
If $I$ and $J$ are finite index sets, write
\begin{equation}\label{eq:mixed-norm-definition}
 \|a\|_{\ell_p(I;\ell_q(J))}
 :=
 \left[
 \sum_{i\in I}
 \left(\sum_{j\in J}|a_{ij}|^q\right)^{p/q}
 \right]^{1/p}.
\end{equation}

Let $m=d+e$, with $d,e\ge1$.  Since
\[
 \frac1{q_r}=\frac12+\frac1{2r},
\]
a direct calculation gives the two barycentric identities
\begin{equation}\label{eq:bary}
 \frac1{q_m}
 =
 \frac dm\frac1{q_d}+\frac em\frac12
 =
 \frac dm\frac12+\frac em\frac1{q_e}.
\end{equation}
Thus $q_m$ lies simultaneously between $q_d$ and $2$, and between $2$ and
$q_e$, with weights given exactly by the relative degrees $d/m$ and $e/m$.

\begin{lemma}
\label{thm:weissler-full}
Let $1\le p\le2$, let $N\ge1$, and let
\[
 H:\mathbb C^N\longrightarrow\mathbb C,
 \qquad
 H(z)=\sum_{\alpha\in\mathcal F}c_\alpha z^\alpha,
 \qquad \mathcal F\subset\mathbb N_0^N \text{ finite},
\]
be an analytic polynomial.  Thus only nonnegative powers of the coordinates
occur; no homogeneity assumption is made.  If $\mathfrak m_N$ denotes
normalized Haar measure on $\mathbb T^N$, then
\begin{equation}\label{eq:weissler-full}
 \left(
 \int_{\mathbb T^N}
 \left|H\!\left(\sqrt{\frac p2}\,z\right)\right|^2
 \,d\mathfrak m_N(z)
 \right)^{1/2}
 \le
 \left(
 \int_{\mathbb T^N}|H(z)|^p\,d\mathfrak m_N(z)
 \right)^{1/p}.
\end{equation}
\end{lemma}

\begin{proof}
We first recall precisely the one-variable input.  If
\[
 h:\mathbb C\longrightarrow\mathbb C,
 \qquad
 h(\zeta)=\sum_{k=0}^{M}c_k\zeta^k,
\]
is an analytic polynomial and $1\le p\le2$, Weissler's analytic
hypercontractive estimate gives
\[
 \left(
 \int_{\mathbb T}|h(\rho\zeta)|^2\,d\mathfrak m_1(\zeta)
 \right)^{1/2}
 \le
 \left(
 \int_{\mathbb T}|h(\zeta)|^p\,d\mathfrak m_1(\zeta)
 \right)^{1/p}
 \qquad\text{whenever}\qquad
 0\le\rho\le\sqrt{\frac p2};
\]
see \cite[Corollary~2.1]{Weissler}.  Taking
$\rho=\sqrt{p/2}$ proves \eqref{eq:weissler-full} when $N=1$ (and for
$p=2$ the assertion is also immediate).  From this point on, to avoid
repeatedly writing the integrals, we use
$\|G\|_{L^s(\mathbb T^k)}:=
(\int_{\mathbb T^k}|G(z)|^s\,d\mathfrak m_k(z))^{1/s}$ and the
corresponding mixed-norm notation when the variables are split into blocks.

We tensorize the one-variable inequality.  Put
\[
 \rho:=\sqrt{\frac p2}
\]
and argue by induction on $N$.  The assertion has just been established
when $N=1$.  Suppose it holds in $N-1$ variables and write
\[
 z=(z',z_N)\in\mathbb T^{N-1}\times\mathbb T.
\]
For each fixed $z_N$, the induction hypothesis applied to the polynomial
$z'\mapsto H(z',\rho z_N)$ gives
\[
 \|H(\rho z',\rho z_N)\|_{L^2_{z'}}
 \le
 \|H(z',\rho z_N)\|_{L^p_{z'}}.
\]
Taking the $L^2$ norm in $z_N$ yields
\begin{equation}\label{eq:weissler-induction-first}
 \|H(\rho z',\rho z_N)\|_{L^2_{z_N}(L^2_{z'})}
 \le
 \|H(z',\rho z_N)\|_{L^2_{z_N}(L^p_{z'})}.
\end{equation}

Because $p\le2$, Minkowski's integral inequality gives the mixed-norm
embedding
\begin{equation}\label{eq:weissler-minkowski}
 \|G\|_{L^2_{z_N}(L^p_{z'})}
 \le
 \|G\|_{L^p_{z'}(L^2_{z_N})}.
\end{equation}
Applying this to the right-hand side of
\eqref{eq:weissler-induction-first}, and then applying the one-variable
Weissler inequality in the last coordinate for each fixed $z'$, we obtain
\begin{align*}
 \|H(\rho z',\rho z_N)\|_{L^2_{z_N}(L^2_{z'})}
 &\le
 \|H(z',\rho z_N)\|_{L^p_{z'}(L^2_{z_N})}\\
 &\le
 \|H(z',z_N)\|_{L^p_{z'}(L^p_{z_N})}\\
 &=
 \|H\|_{L^p(\mathbb T^N)}.
\end{align*}
This completes the induction.
\end{proof}

\begin{corollary}
\label{lem:hyper}
Let $R$ be an $r$-homogeneous analytic polynomial on $\mathbb C^N$.  If
$1\le p\le2$, then
\begin{equation}\label{eq:homogeneous-hypercontractivity}
 \|R\|_{L^2(\mathbb T^N)}
 \le
 \left(\frac2p\right)^{r/2}
 \|R\|_{L^p(\mathbb T^N)}.
\end{equation}
\end{corollary}

\begin{proof}
Apply \Cref{thm:weissler-full} and use homogeneity:
\[
 R\!\left(\sqrt{\frac p2}\,z\right)
 =
 \left(\frac p2\right)^{r/2}R(z).
\]
\end{proof}

\begin{lemma}
\label{lem:mixed-norm-interpolation}
Let $I,J$ be finite index sets, let
\[
 1\le p_0,p_1,q_0,q_1<\infty,
 \qquad
 0\le\vartheta\le1,
\]
and define $p,q$ by
\begin{equation}\label{eq:mixed-interpolation-exponents}
 \frac1p
 =
 \frac{1-\vartheta}{p_0}+\frac{\vartheta}{p_1},
 \qquad
 \frac1q
 =
 \frac{1-\vartheta}{q_0}+\frac{\vartheta}{q_1}.
\end{equation}
Then every scalar array $a=(a_{ij})_{I\times J}$ satisfies
\begin{equation}\label{eq:mixed-norm-interpolation}
 \|a\|_{\ell_p(I;\ell_q(J))}
 \le
 \|a\|_{\ell_{p_0}(I;\ell_{q_0}(J))}^{1-\vartheta}
 \|a\|_{\ell_{p_1}(I;\ell_{q_1}(J))}^{\vartheta}.
\end{equation}
\end{lemma}

\begin{proof}
The cases $\vartheta=0$ and $\vartheta=1$ are immediate, so assume
$0<\vartheta<1$.  For each $i\in I$, log-convexity of finite
$\ell_q$ norms gives
\[
 \|(a_{ij})_{j\in J}\|_q
 \le
 \|(a_{ij})_{j\in J}\|_{q_0}^{1-\vartheta}
 \|(a_{ij})_{j\in J}\|_{q_1}^{\vartheta}.
\]
Raising to the power $p$ and summing over $i$ gives
\[
 \|a\|_{\ell_p(I;\ell_q(J))}^p
 \le
 \sum_{i\in I}
 b_i^{p(1-\vartheta)}c_i^{p\vartheta},
\]
where
\[
 b_i:=\|(a_{ij})_{j\in J}\|_{q_0},
 \qquad
 c_i:=\|(a_{ij})_{j\in J}\|_{q_1}.
\]
The exponents
\[
 r_0:=\frac{p_0}{p(1-\vartheta)},
 \qquad
 r_1:=\frac{p_1}{p\vartheta}
\]
are conjugate, because \eqref{eq:mixed-interpolation-exponents} implies
\[
 \frac1{r_0}+\frac1{r_1}
 =
 \frac{p(1-\vartheta)}{p_0}
 +
 \frac{p\vartheta}{p_1}
 =1.
\]
H\"older's inequality therefore yields
\begin{align*}
 \|a\|_{\ell_p(I;\ell_q(J))}^p
 &\le
 \left(\sum_{i\in I}b_i^{p_0}\right)^{p(1-\vartheta)/p_0}
 \left(\sum_{i\in I}c_i^{p_1}\right)^{p\vartheta/p_1}.
\end{align*}
Taking the $p$th root proves
\eqref{eq:mixed-norm-interpolation}.
\end{proof}

\begin{lemma}
\label{lem:Blei}
Let $d,e\ge1$, set $m=d+e$, and let
$a=(a_{\beta\gamma})$ be a finite scalar array.  This is the two-block form
of Blei's inequality~\cite{Blei}. Define
\begin{equation}\label{eq:Xd-definition}
 X_d(a)
 :=
 \left[
 \sum_\beta
 \left(\sum_\gamma|a_{\beta\gamma}|^2\right)^{q_d/2}
 \right]^{1/q_d}
\end{equation}
and
\begin{equation}\label{eq:Ye-definition}
 Y_e(a)
 :=
 \left[
 \sum_\gamma
 \left(\sum_\beta|a_{\beta\gamma}|^2\right)^{q_e/2}
 \right]^{1/q_e}.
\end{equation}
Then
\begin{equation}\label{eq:two-block-Blei}
 \left(\sum_{\beta,\gamma}|a_{\beta\gamma}|^{q_m}\right)^{1/q_m}
 \le
 X_d(a)^{d/m}Y_e(a)^{e/m}.
\end{equation}
\end{lemma}

\begin{proof}
Apply \Cref{lem:mixed-norm-interpolation} with
\[
 (p_0,q_0)=(q_d,2),
 \qquad
 (p_1,q_1)=(2,q_e),
 \qquad
 \vartheta=\frac em.
\]
By the two identities in \eqref{eq:bary}, both interpolated exponents are
equal to $q_m$.  Hence
\begin{align}
 \|a\|_{\ell_{q_m}(\beta;\ell_{q_m}(\gamma))}
 &\le
 \|a\|_{\ell_{q_d}(\beta;\ell_2(\gamma))}^{d/m}
 \|a\|_{\ell_2(\beta;\ell_{q_e}(\gamma))}^{e/m}.
 \label{eq:two-block-before-Minkowski}
\end{align}
The first norm on the right is exactly $X_d(a)$.  Since $q_e\le2$,
Minkowski's inequality gives
\begin{align}
 \|a\|_{\ell_2(\beta;\ell_{q_e}(\gamma))}
 &=
 \left[
 \sum_\beta
 \left(\sum_\gamma|a_{\beta\gamma}|^{q_e}\right)^{2/q_e}
 \right]^{1/2}
 \notag\\
 &\le
 \left[
 \sum_\gamma
 \left(\sum_\beta|a_{\beta\gamma}|^2\right)^{q_e/2}
 \right]^{1/q_e}
 =
 Y_e(a).
 \label{eq:two-block-Minkowski}
\end{align}
Finally,
\[
 \|a\|_{\ell_{q_m}(\beta;\ell_{q_m}(\gamma))}
 =
 \left(\sum_{\beta,\gamma}|a_{\beta\gamma}|^{q_m}\right)^{1/q_m},
\]
because the inner and outer exponents coincide.  Combining
\eqref{eq:two-block-before-Minkowski} and
\eqref{eq:two-block-Minkowski} proves \eqref{eq:two-block-Blei}.
\end{proof}

\begin{lemma}
\label{thm:bihom}
Let $d,e\ge1$, set $m=d+e$, and let
\[
 Q(x,y)
 =
 \sum_{\substack{|\beta|=d\\|\gamma|=e}}
 a_{\beta\gamma}x^\beta y^\gamma
\]
be bihomogeneous on
$\mathbb C^{n_x}\times\mathbb C^{n_y}$.  Then
\begin{align}
 \|a(Q)\|_{q_m}
 &\le
 \left(1+\frac1d\right)^{de/(2m)}
 \left(1+\frac1e\right)^{de/(2m)}
 D_d^{d/m}D_e^{e/m}\|Q\|_\infty
 \label{eq:bihom-sharp}\\
 &\le
 \exp(1/2)
 D_d^{d/m}D_e^{e/m}\|Q\|_\infty.
 \label{eq:bihom-simple}
\end{align}
\end{lemma}

\begin{proof}
For each multi-index $\beta$ with $|\beta|=d$, define the
$e$-homogeneous polynomial
\[
 A_\beta(y)
 :=
 \sum_{|\gamma|=e}a_{\beta\gamma}y^\gamma.
\]
Parseval's identity gives
\[
 \left(\sum_\gamma|a_{\beta\gamma}|^2\right)^{1/2}
 =
 \|A_\beta\|_{L^2(\mathbb T^{n_y})}.
\]
Apply \Cref{lem:hyper} to $A_\beta$ with $p=q_d$.  Since
\[
 \frac2{q_d}=1+\frac1d,
\]
we obtain
\begin{equation}\label{eq:row-hypercontractive}
 \left(\sum_\gamma|a_{\beta\gamma}|^2\right)^{1/2}
 \le
 \left(1+\frac1d\right)^{e/2}
 \|A_\beta\|_{L^{q_d}(\mathbb T^{n_y})}.
\end{equation}
Raise \eqref{eq:row-hypercontractive} to the power $q_d$, sum in
$\beta$, and use Tonelli's theorem:
\begin{align}
 X_d(a)^{q_d}
 &\le
 \left(1+\frac1d\right)^{eq_d/2}
 \int_{\mathbb T^{n_y}}
 \sum_{|\beta|=d}|A_\beta(y)|^{q_d}\,d\mathfrak m_{n_y}(y).
 \label{eq:Xd-integrated}
\end{align}

For each fixed $y\in\mathbb T^{n_y}$, the polynomial
\[
 x\longmapsto Q(x,y)
 =
 \sum_{|\beta|=d}A_\beta(y)x^\beta
\]
is $d$-homogeneous.  The degree-$d$ polynomial Bohnenblust--Hille
inequality therefore gives
\[
 \left(\sum_{|\beta|=d}|A_\beta(y)|^{q_d}\right)^{1/q_d}
 \le
 D_d\sup_{x\in\mathbb D^{n_x}}|Q(x,y)|
 \le
 D_d\|Q\|_\infty.
\]
Substituting this pointwise estimate into \eqref{eq:Xd-integrated} and
using the normalization of Haar measure yields
\begin{equation}\label{eq:Xd-bound}
 X_d(a)
 \le
 \left(1+\frac1d\right)^{e/2}
 D_d\|Q\|_\infty.
\end{equation}

Interchanging the two variable blocks gives, in exactly the same way,
\begin{equation}\label{eq:Ye-bound}
 Y_e(a)
 \le
 \left(1+\frac1e\right)^{d/2}
 D_e\|Q\|_\infty.
\end{equation}
Combining \Cref{lem:Blei} with
\eqref{eq:Xd-bound} and \eqref{eq:Ye-bound}, we find
\begin{align*}
 \|a(Q)\|_{q_m}
 &\le
 \left[
 \left(1+\frac1d\right)^{e/2}
 D_d\|Q\|_\infty
 \right]^{d/m}\\
 &\qquad\times
 \left[
 \left(1+\frac1e\right)^{d/2}
 D_e\|Q\|_\infty
 \right]^{e/m}\\
 &=
 \left(1+\frac1d\right)^{de/(2m)}
 \left(1+\frac1e\right)^{de/(2m)}
 D_d^{d/m}D_e^{e/m}\|Q\|_\infty,
\end{align*}
which proves \eqref{eq:bihom-sharp}.

To obtain the uniform form, use $\log(1+s)\le s$ for $s\ge0$:
\begin{align*}
 &\log\left[
 \left(1+\frac1d\right)^{de/(2m)}
 \left(1+\frac1e\right)^{de/(2m)}
 \right]\\
 &\qquad=
 \frac{de}{2m}
 \left[
 \log\left(1+\frac1d\right)
 +
 \log\left(1+\frac1e\right)
 \right]\\
 &\qquad\le
 \frac{de}{2m}\left(\frac1d+\frac1e\right)
 =
 \frac12.
\end{align*}
Exponentiating proves \eqref{eq:bihom-simple}.
\end{proof}

\subsection{Fixed-ratio reduction and iteration}
\label{sec:hybrid}

A fixed-ratio reduction already improves the earlier subexponential scale to a
quasipolynomial bound.  The polynomial estimate will require the weighted form
developed later.

The two-block estimate becomes effective only after the coefficient array
has been distributed among bidegrees that remain uniformly separated from
$0$ and $m$.  A coloring of whole coordinates by $x$ and $y$ achieves this
for diffuse monomials, but it can fail completely when one coordinate
carries a large fraction of the degree.  A third, rare substitution is therefore
allowed: a coordinate may be replaced by $(x_j+y_j)/2$.

The probability of this third choice is calibrated to the degree.  An
$m$-homogeneous monomial has at most $m$ active coordinates, so choosing
the splitting probability to be $1/m$ leaves all active coordinates
unsplit with probability bounded away from zero.  At the same time, a
prescribed dominant coordinate is split with probability of order $1/m$.
This balance produces the polynomial recovery cost.

Fix $m\ge2$ and an $m$-homogeneous polynomial
\[
 P(z)=\sum_{|\alpha|=m}a_\alpha z^\alpha
\]
on $\mathbb C^n$.  Put
\[
 \vartheta:=\frac1m.
\]
Independently for each coordinate $j$, make the substitution
\begin{equation}\label{eq:subs}
 z_j\longmapsto
 \begin{cases}
 x_j,&\text{with probability }(1-\vartheta)/2,\\[2mm]
 y_j,&\text{with probability }(1-\vartheta)/2,\\[2mm]
 (x_j+y_j)/2,&\text{with probability }\vartheta.
 \end{cases}
\end{equation}
Let $\Phi_\omega$ denote the resulting random map and set
\[
 Q_\omega(x,y):=P(\Phi_\omega(x,y)).
\]
Since every substitution in \eqref{eq:subs} is linear, $Q_\omega$ is
$m$-homogeneous in the joint variables $(x,y)$.  We write
$Q_{\omega,d}$ for its component of degree $d$ in $x$ and degree $m-d$
in $y$.

For a fixed outcome $\omega$, denote by
\[
 X_\omega,\qquad Y_\omega,\qquad S_\omega
\]
the sets of coordinates sent to $x_j$, sent to $y_j$, and split,
respectively.  If $|\alpha|=m$, then
\[
 z^\alpha\circ\Phi_\omega
 =
 \sum_{\beta+\gamma=\alpha}
 c^\omega_{\alpha,\beta}x^\beta y^\gamma,
\]
where
\begin{equation}\label{eq:child-coefficient}
 c^\omega_{\alpha,\beta}
 =
 2^{-\sum_{j\in S_\omega}\alpha_j}
 \prod_{j\in S_\omega}\binom{\alpha_j}{\beta_j},
\end{equation}
provided that
\[
 \beta_j=\alpha_j \quad(j\in X_\omega),
 \qquad
 \beta_j=0 \quad(j\in Y_\omega),
 \qquad
 0\le\beta_j\le\alpha_j \quad(j\in S_\omega);
\]
outside these conditions we set
$c^\omega_{\alpha,\beta}=0$.  Thus the unsplit coordinates are assigned
entirely to one block, while the split coordinates retain their exact
binomial expansion.

\begin{lemma}
\label{lem:recovery}
For every outcome $\omega$ and every $0\le d\le m$,
\begin{equation}\label{eq:component-contractivity}
 \|Q_{\omega,d}\|_\infty\le\|P\|_\infty.
\end{equation}
Moreover, for every $I\subset\{0,\ldots,m\}$,
\begin{equation}\label{eq:exactmass}
 \sum_{d\in I}\|a(Q_{\omega,d})\|_{q_m}^{q_m}
 =
 \sum_{|\alpha|=m}|a_\alpha|^{q_m}
 \sum_{\substack{\beta+\gamma=\alpha\\|\beta|\in I}}
 |c^\omega_{\alpha,\beta}|^{q_m}.
\end{equation}
\end{lemma}

\begin{proof}
If $(x,y)\in\mathbb D^n\times\mathbb D^n$, then every coordinate of
$\Phi_\omega(x,y)$ belongs to $\mathbb D$: this is immediate for the
unsplit substitutions, while
\[
 \left|\frac{x_j+y_j}{2}\right|
 \le\frac{|x_j|+|y_j|}{2}\le1
\]
for a split coordinate.  Hence
\[
 \|Q_\omega\|_\infty\le\|P\|_\infty.
\]
Projection onto the component of $x$-degree $d$ is given by
\begin{equation}\label{eq:x-degree-projection}
 Q_{\omega,d}(x,y)
 =
 \frac1{2\pi}
 \int_0^{2\pi}
 Q_\omega(\e^{it}x,y)\e^{-idt}\,dt.
\end{equation}
It is therefore an average of isometries, which proves
\eqref{eq:component-contractivity}.

For the coefficient identity, fix a child pair $(\beta,\gamma)$.  Its
parent is necessarily
\[
 \alpha=\beta+\gamma.
\]
Thus the coefficient of $x^\beta y^\gamma$ in $Q_\omega$ is exactly
\begin{equation}\label{eq:exact-child-coefficient}
 a_{\beta+\gamma}
 c^\omega_{\beta+\gamma,\beta};
\end{equation}
there is no sum over distinct parent coefficients.  In particular, no
cancellation or collision occurs before absolute values are taken.
Summing the $q_m$-powers of
\eqref{eq:exact-child-coefficient} over the bidegrees with
$|\beta|\in I$ gives \eqref{eq:exactmass}.
\end{proof}

For $0<\varepsilon<1/2$, define the central window
\begin{equation}\label{eq:central-window}
 I_{m,\varepsilon}
 :=
 \left\{
 d\in\{0,\ldots,m\}:
 \varepsilon m\le d\le(1-\varepsilon)m
 \right\}.
\end{equation}
For the binomial mass created by a split dominant coordinate, by a central integer of
$\{0,\ldots,a\}$ we mean either $\lfloor a/2\rfloor$ or
$\lceil a/2\rceil$.

\begin{lemma}
\label{lem:half-mass}
Let $(w_b)_{b=0}^a$ be a nonnegative sequence that is symmetric about
$a/2$ and nondecreasing up to the midpoint.  If $K$ is an interval of
consecutive integers in $\{0,\ldots,a\}$ that contains a central integer
and satisfies
\[
 |K|\ge\left\lceil\frac{a+1}{2}\right\rceil,
\]
then
\begin{equation}\label{eq:half-mass}
 \sum_{b\in K}w_b
 \ge
 \frac12\sum_{b=0}^a w_b.
\end{equation}
\end{lemma}

\begin{proof}
Put $k:=|K|$.  Among all intervals of $k$ consecutive indices, the mass
of a symmetric unimodal sequence is minimized at one of the two extreme
intervals
\[
 L_k:=\{0,\ldots,k-1\},
 \qquad
 R_k:=\{a-k+1,\ldots,a\}.
\]
Indeed, shifting an interval one step toward the midpoint replaces an
outer weight by an inner weight and cannot decrease its mass.

By symmetry,
\[
 \sum_{b\in L_k}w_b=\sum_{b\in R_k}w_b.
\]
Since $2k\ge a+1$, the union $L_k\cup R_k$ contains every index in
$\{0,\ldots,a\}$.  Consequently,
\[
 2\sum_{b\in L_k}w_b
 =
 \sum_{b\in L_k}w_b+\sum_{b\in R_k}w_b
 \ge
 \sum_{b=0}^a w_b.
\]
Every admissible interval has mass at least that of $L_k$, which proves
\eqref{eq:half-mass}.
\end{proof}

\begin{lemma}
\label{thm:capture}
Let
\begin{equation}\label{eq:epsilon-star}
 \varepsilon_*:=\frac{3-\sqrt5}{4}.
\end{equation}
For every $0<\varepsilon<\varepsilon_*$ there exist
$c_\varepsilon>0$ and $m_\varepsilon\in\mathbb N$ such that, for every
$m\ge m_\varepsilon$, every $n$, and every
$\alpha\in\mathbb N_0^n$ with $|\alpha|=m$,
\begin{equation}\label{eq:capture}
 \mathbb E_\omega
 \sum_{\substack{\beta+\gamma=\alpha\\
                 |\beta|\in I_{m,\varepsilon}}}
 |c^\omega_{\alpha,\beta}|^{q_m}
 \ge
 c_\varepsilon m^{-3/2}.
\end{equation}
The constants are independent of the ambient dimension and of the
distribution of the degree among the coordinates.
\end{lemma}

\begin{proof}
Write the positive entries of $\alpha$ as
\[
 a_1,\ldots,a_r,
 \qquad
 a_1+\cdots+a_r=m,
\]
and put
\[
 a:=\max_{1\le j\le r}a_j.
\]
In particular, $r\le m$.  We distinguish whether the degree is diffuse
or concentrated in one coordinate.

\smallskip
\noindent
\emph{Step 1: diffuse parents.}
Assume that
\[
 a\le2\varepsilon m.
\]
Consider the event $E_0$ that none of the $r$ active coordinates is
split.  Since $r\le m$,
\begin{equation}\label{eq:no-split-probability}
 \mathbb P(E_0)
 =
 \left(1-\frac1m\right)^r
 \ge
 \left(1-\frac1m\right)^m
 \ge\frac14.
\end{equation}
Conditional on $E_0$, the active coordinates are colored independently
by $x$ and $y$, each with probability $1/2$.  The resulting $x$-degree is
\[
 S:=\sum_{j=1}^r a_j\xi_j,
 \qquad
 \mathbb P(\xi_j=0)=\mathbb P(\xi_j=1)=\frac12.
\]
Therefore
\begin{equation}\label{eq:diffuse-mean-variance}
 \mathbb ES=\frac m2,
 \qquad
 \operatorname{Var}(S)
 =
 \frac14\sum_{j=1}^r a_j^2
 \le
 \frac{a}{4}\sum_{j=1}^r a_j
 =
 \frac{am}{4}
 \le
 \frac{\varepsilon m^2}{2}.
\end{equation}
The distance from $m/2$ to the complement of the central interval
$[\varepsilon m,(1-\varepsilon)m]$ is
$(1/2-\varepsilon)m$.  Chebyshev's inequality and
\eqref{eq:diffuse-mean-variance} give
\begin{equation}\label{eq:diffuse-Chebyshev}
 \mathbb P\{S\notin I_{m,\varepsilon}\mid E_0\}
 \le
 \frac{2\varepsilon}{(1-2\varepsilon)^2}.
\end{equation}
Set
\begin{equation}\label{eq:delta-epsilon}
 \delta_\varepsilon
 :=
 1-\frac{2\varepsilon}{(1-2\varepsilon)^2}.
\end{equation}
For $0<\varepsilon<1/2$,
\[
 \delta_\varepsilon>0
 \quad\Longleftrightarrow\quad
 4\varepsilon^2-6\varepsilon+1>0
 \quad\Longleftrightarrow\quad
 \varepsilon<\frac{3-\sqrt5}{4}.
\]
Thus $\delta_\varepsilon>0$ under the hypothesis of the theorem.

On $E_0$, every parent monomial produces a single child, whose coefficient
has modulus one.  Hence
\begin{equation}\label{eq:diffuse-capture}
 \mathbb E_\omega
 \sum_{\substack{\beta+\gamma=\alpha\\
                 |\beta|\in I_{m,\varepsilon}}}
 |c^\omega_{\alpha,\beta}|^{q_m}
 \ge
 \frac{\delta_\varepsilon}{4}.
\end{equation}

\smallskip
\noindent
\emph{Step 2: a dominant coordinate.}
Assume now that
\[
 a>2\varepsilon m,
\]
and choose $j_*$ so that $a_{j_*}=a$.  Let $E_*$ be the event that
$j_*$ is split and every other active coordinate is unsplit.  Its
probability satisfies
\begin{align}
 \mathbb P(E_*)
 &=
 \frac1m
 \left(1-\frac1m\right)^{r-1} \notag\\
 &\ge
 \frac1m
 \left(1-\frac1m\right)^{m-1}
 \ge
 \frac1{4m}.
 \label{eq:dominant-event-probability}
\end{align}

Fix the coloring of the unsplit coordinates, and let
\[
 s_0\in[0,m-a]
\]
be their total $x$-degree.  The dominant coordinate contributes
\[
 2^{-a}\sum_{b=0}^a
 \binom ab x_{j_*}^b y_{j_*}^{a-b}.
\]
The values of $b$ for which the total $x$-degree belongs to the central
window form the integer interval
\begin{equation}\label{eq:dominant-admissible-set}
 K:=J\cap\mathbb Z,
 \qquad
 J:=
 [\varepsilon m-s_0,(1-\varepsilon)m-s_0]\cap[0,a].
\end{equation}
The interval $K$ satisfies the hypotheses of
\Cref{lem:half-mass}.  Since $s_0\ge0$ and $a>2\varepsilon m$,
\[
 \varepsilon m-s_0
 \le\varepsilon m
 <\frac a2.
\]
On the other hand, $s_0\le m-a$ gives
\[
 (1-\varepsilon)m-s_0
 \ge a-\varepsilon m
 >\frac a2.
\]
Thus the real interval $J$ contains the midpoint $a/2$.

If its left endpoint is truncated at $0$, then
\[
 J=[0,U]
 \qquad\text{with}\qquad U>\frac a2.
\]
It follows directly that $J\cap\mathbb Z$ contains a central integer and
at least
\[
 \left\lceil\frac{a+1}{2}\right\rceil
\]
indices.  The case in which the right endpoint is truncated at $a$ is
symmetric.

Suppose finally that neither endpoint is truncated.  Then
\begin{align}
 |J|_{\mathbb R}
 &=
 (1-2\varepsilon)m \notag\\
 &\ge
 \frac a2+\left(\frac12-2\varepsilon\right)m,
 \label{eq:dominant-interval-length}
\end{align}
because $a\le m$.  Since
$\varepsilon<\varepsilon_*<1/4$, we may choose $m_\varepsilon$ so large
that
\[
 \left(\frac12-2\varepsilon\right)m>1
 \qquad(m\ge m_\varepsilon).
\]
Hence $|J|_{\mathbb R}>a/2+1$.  Every closed interval of length $L$
contains at least $\lfloor L\rfloor$ integers, so a direct distinction
between even and odd $a$ gives
\[
 |K|
 \ge
 \left\lceil\frac{a+1}{2}\right\rceil.
\]
Moreover, because $J$ contains $a/2$ and has length greater than one, it
contains at least one of the two nearest integers
$\lfloor a/2\rfloor$ and $\lceil a/2\rceil$.

Apply \Cref{lem:half-mass} to
\[
 w_b:=\binom ab^{q_m}.
\]
This sequence is symmetric and unimodal.  If
\[
 p_b:=2^{-a}\binom ab,
\]
then
\begin{equation}\label{eq:dominant-half-mass}
 \sum_{b\in K}p_b^{q_m}
 \ge
 \frac12\sum_{b=0}^a p_b^{q_m}.
\end{equation}
Since $0\le p_b\le1$ and $q_m\le2$,
\begin{align}
 \sum_{b=0}^a p_b^{q_m}
 &\ge
 \sum_{b=0}^a p_b^2 \notag\\
 &=
 4^{-a}\sum_{b=0}^a\binom ab^2 \notag\\
 &=
 4^{-a}\binom{2a}{a}
 \ge
 \frac1{2\sqrt a}
 \ge
 \frac1{2\sqrt m}.
 \label{eq:central-binomial-lower}
\end{align}
The penultimate inequality follows by induction. If
\[
 u_a:=4^{-a}\binom{2a}{a},
\]
then $u_1=1/2$ and
\[
 \frac{u_{a+1}}{u_a}
 =
 \frac{2a+1}{2a+2}
 \ge
 \sqrt{\frac a{a+1}}.
\]

Combining \eqref{eq:dominant-event-probability},
\eqref{eq:dominant-half-mass}, and
\eqref{eq:central-binomial-lower}, we obtain
\begin{equation}\label{eq:dominant-capture}
 \mathbb E_\omega
 \sum_{\substack{\beta+\gamma=\alpha\\
                 |\beta|\in I_{m,\varepsilon}}}
 |c^\omega_{\alpha,\beta}|^{q_m}
 \ge
 \frac1{16m^{3/2}}.
\end{equation}

The diffuse estimate \eqref{eq:diffuse-capture} and the dominant estimate
\eqref{eq:dominant-capture} together prove \eqref{eq:capture}.  For
example, after increasing $m_\varepsilon$ as above, one may take
\[
 c_\varepsilon
 :=
 \min\left\{
 \frac{\delta_\varepsilon}{4},
 \frac1{16}
 \right\}.
\]
\end{proof}

\begin{lemma}
\label{thm:recurrence}
Define the monotone envelope
\begin{equation}\label{eq:D-envelope}
 \mathcal D_m:=\max_{1\le j\le m}D_j.
\end{equation}
For every $0<\varepsilon<\varepsilon_*$ there is
$C_\varepsilon<\infty$ such that, with
\[
 \rho:=1-\varepsilon,
\]
one has
\begin{equation}\label{eq:fixed-ratio-recurrence}
 D_m
 \le
 C_\varepsilon m^{5/(2q_m)}
 \mathcal D_{\lceil\rho m\rceil}
\end{equation}
for all sufficiently large $m$.
\end{lemma}

\begin{proof}
Let $P$ be an arbitrary $m$-homogeneous polynomial.  Taking expectations
in \eqref{eq:exactmass}, using Tonelli's theorem, and then applying
\Cref{thm:capture} parent by parent gives
\begin{align}
 &\mathbb E_\omega
 \sum_{d\in I_{m,\varepsilon}}
 \|a(Q_{\omega,d})\|_{q_m}^{q_m} \notag\\
 &\qquad=
 \sum_{|\alpha|=m}|a_\alpha|^{q_m}
 \mathbb E_\omega
 \sum_{\substack{\beta+\gamma=\alpha\\
                 |\beta|\in I_{m,\varepsilon}}}
 |c^\omega_{\alpha,\beta}|^{q_m} \notag\\
 &\qquad\ge
 c_\varepsilon m^{-3/2}
 \sum_{|\alpha|=m}|a_\alpha|^{q_m}.
 \label{eq:recovery-averaged-lower}
\end{align}
Equivalently,
\begin{equation}\label{eq:recovery-averaged}
 \|a(P)\|_{q_m}^{q_m}
 \le
 c_\varepsilon^{-1}m^{3/2}
 \mathbb E_\omega
 \sum_{d\in I_{m,\varepsilon}}
 \|a(Q_{\omega,d})\|_{q_m}^{q_m}.
\end{equation}

If $d\in I_{m,\varepsilon}$, then
\[
 1\le d\le\rho m,
 \qquad
 1\le m-d\le\rho m
\]
for all sufficiently large $m$.  Hence
\[
 D_d,\ D_{m-d}
 \le
 \mathcal D_{\lceil\rho m\rceil}.
\]
Applying \Cref{thm:bihom} to $Q_{\omega,d}$ and then
\Cref{lem:recovery}, we obtain
\begin{align}
 \|a(Q_{\omega,d})\|_{q_m}
 &\le
 \exp(1/2)
 D_d^{d/m}
 D_{m-d}^{(m-d)/m}
 \|Q_{\omega,d}\|_\infty \notag\\
 &\le
 \exp(1/2)
 \mathcal D_{\lceil\rho m\rceil}
 \|P\|_\infty.
 \label{eq:central-bidegree-bound}
\end{align}
The exponents $d/m$ and $(m-d)/m$ add to one; this is the precise point
at which the fractional form of the two-block theorem becomes decisive.

The central window contains at most $m$ bidegrees.  Substituting
\eqref{eq:central-bidegree-bound} into
\eqref{eq:recovery-averaged} yields
\[
 \|a(P)\|_{q_m}^{q_m}
 \le
 c_\varepsilon^{-1}
 m^{5/2}
 \exp(q_m/2)
 \mathcal D_{\lceil\rho m\rceil}^{q_m}
 \|P\|_\infty^{q_m}.
\]
Taking the $q_m$-th root gives
\[
 \|a(P)\|_{q_m}
 \le
 \exp(1/2)c_\varepsilon^{-1/q_m}
 m^{5/(2q_m)}
 \mathcal D_{\lceil\rho m\rceil}
 \|P\|_\infty.
\]
Since $q_m\ge4/3$ for $m\ge2$ and $c_\varepsilon$ is fixed, the first two
factors are bounded by a constant depending only on $\varepsilon$.
Taking the supremum over $P$ proves
\eqref{eq:fixed-ratio-recurrence}.
\end{proof}

\begin{lemma}\label{lem:fixedratio}
Let $(A_m)_{m\ge1}$ be a positive sequence and define
\[
 \mathcal A_m:=\max_{1\le j\le m}A_j.
\]
Suppose that there exist $0<\rho<1$, $a>0$, and $C<\infty$ such that
\begin{equation}\label{eq:abstract-fixed-ratio}
 A_m\le C m^a\mathcal A_{\lceil\rho m\rceil}
\end{equation}
for all sufficiently large $m$.  Then
\begin{equation}\label{eq:abstract-fixed-ratio-conclusion}
 \log\mathcal A_m
 \le
 \frac{a}{2\log(1/\rho)}(\log m)^2
 +O_{\rho,a,C}(\log m).
\end{equation}
The implicit constant may also depend on the finite initial segment of
the sequence preceding the range in which
\eqref{eq:abstract-fixed-ratio} holds.
\end{lemma}

\begin{proof}
Choose $M_0$ so large that \eqref{eq:abstract-fixed-ratio} holds for
$m>M_0$ and
\[
 \lceil\rho m\rceil<m
 \qquad(m>M_0).
\]
After increasing the constant to absorb the finitely many indices
$j\le M_0$, taking the maximum over $1\le j\le m$ gives
\begin{equation}\label{eq:envelope-fixed-ratio}
 \mathcal A_m
 \le
 C_0m^a\mathcal A_{\lceil\rho m\rceil},
 \qquad m>M_0,
\end{equation}
for some $C_0\ge1$.  Indeed, for $M_0<j\le m$,
\[
 A_j
 \le
 Cj^a\mathcal A_{\lceil\rho j\rceil}
 \le
 Cm^a\mathcal A_{\lceil\rho m\rceil},
\]
whereas the finitely many remaining values are absorbed into $C_0$.

Set
\[
 m_0:=m,
 \qquad
 m_{j+1}:=\lceil\rho m_j\rceil,
\]
and stop at the first index $J$ for which $m_J\le M_0$.  Iterating
\eqref{eq:envelope-fixed-ratio} gives
\begin{equation}\label{eq:fixed-ratio-log-iteration}
 \log\mathcal A_m
 \le
 \log\mathcal A_{m_J}
 +J\log C_0
 +a\sum_{j=0}^{J-1}\log m_j.
\end{equation}

The ceiling introduces only a bounded error.  More precisely, induction
from
\[
 \rho m_j\le m_{j+1}\le\rho m_j+1
\]
gives
\begin{equation}\label{eq:geometric-degree-sequence}
 \rho^jm
 \le
 m_j
 \le
 \rho^jm+\frac{1}{1-\rho}.
\end{equation}
Consequently,
\begin{equation}\label{eq:number-generations}
 J
 =
 \frac{\log m}{\log(1/\rho)}
 +O_\rho(1).
\end{equation}
Moreover, for $j<J$, \eqref{eq:geometric-degree-sequence} implies
\[
 \log m_j
 =
 \log m+j\log\rho+O_\rho(1).
\]
Writing $L:=\log(1/\rho)$ and summing this arithmetic progression, we
obtain
\begin{align}
 \sum_{j=0}^{J-1}\log m_j
 &=
 J\log m
 -L\frac{J(J-1)}2
 +O_\rho(J) \notag\\
 &=
 \frac{(\log m)^2}{2L}
 +O_\rho(\log m).
 \label{eq:fixed-ratio-log-sum}
\end{align}
Since $m_J$ remains in a fixed finite set, the first term on the
right-hand side of \eqref{eq:fixed-ratio-log-iteration} is bounded.
Substituting \eqref{eq:number-generations} and
\eqref{eq:fixed-ratio-log-sum} into
\eqref{eq:fixed-ratio-log-iteration} proves
\eqref{eq:abstract-fixed-ratio-conclusion}.
\end{proof}

The iteration in \Cref{lem:fixedratio} is a discrete renormalization in the
degree: the fixed reduction $m\mapsto\lceil\rho m\rceil$ propagates an
estimate along geometric scales. This is formally analogous to fixed-scale
improvement iterations in nonlinear elliptic regularity; compare
\cite[Section~2]{TeixeiraUniversalModuli}. Here the accumulated loss is
quadratic in $\log m$; the analogy is heuristic and no PDE result enters
the proof.

\begin{proposition}
\label{thm:quasipoly}
The optimal complex polynomial Bohnenblust--Hille constants satisfy
\begin{equation}\label{eq:quasipolynomial-limsup}
 \limsup_{m\to\infty}
 \frac{\log D_m}{(\log m)^2}
 \le
 \frac{5}
 {8\log\!\bigl(4/(1+\sqrt5)\bigr)}
 =
 2.949012\ldots<2.95.
\end{equation}
Equivalently,
\begin{equation}\label{eq:quasipolynomial-explicit-form}
 D_m
 \le
 \exp\!\left(
 \left[
 \frac{5}
 {8\log\!\bigl(4/(1+\sqrt5)\bigr)}
 +o(1)
 \right](\log m)^2
 \right).
\end{equation}
\end{proposition}

\begin{proof}
Since
\[
 q_m=\frac{2m}{m+1},
\]
we have
\begin{equation}\label{eq:one-step-exponent}
 \frac5{2q_m}
 =
 \frac54+\frac{5}{4m}.
\end{equation}
Moreover,
\[
 m^{5/(4m)}
 =
 \exp\!\left(\frac{5\log m}{4m}\right)
 \le
 \exp\!\left(\frac{5}{4\e}\right),
\]
because $\sup_{x>0}(\log x)/x=1/\e$.  Thus the factor
$m^{5/(2q_m)}$ in \eqref{eq:fixed-ratio-recurrence} may be replaced,
after enlarging $C_\varepsilon$, by $m^{5/4}$.

Apply \Cref{lem:fixedratio} to the sequence $A_m=D_m$, with
\[
 a=\frac54,
 \qquad
 \rho=1-\varepsilon.
\]
It follows that
\[
 \limsup_{m\to\infty}
 \frac{\log\mathcal D_m}{(\log m)^2}
 \le
 \frac{5}{8\log(1/(1-\varepsilon))}.
\]
Since $D_m\le\mathcal D_m$, the same bound holds for $D_m$.  Finally,
let $\varepsilon\uparrow\varepsilon_*$.  By
\eqref{eq:epsilon-star},
\[
 1-\varepsilon_*
 =
 \frac{1+\sqrt5}{4},
\]
and therefore
\[
 \log\frac1{1-\varepsilon_*}
 =
 \log\!\left(\frac4{1+\sqrt5}\right).
\]
\end{proof}

The fixed-ratio reduction therefore improves the classical subexponential
scale to a quasipolynomial one.  Its iteration, however, accumulates a
polynomial cost at every generation.  The weighted argument below replaces
this pointwise iteration by a simultaneous estimate across all homogeneous
levels; this is the step that removes the accumulated loss and yields a
polynomial bound.

\section{The weighted polynomial bootstrap}
\label{sec:polynomial-bootstrap}

\subsection{Weighted constants and central bidegrees}

We now replace the pointwise iteration of degree reductions by a weighted
square-sum estimate across all homogeneous levels.

For $B>0$ and $M\ge1$, let $\Gamma_M(B)$ be the least constant such that
\begin{equation}\label{eq:Gamma-bootstrap}
 \left[
 \sum_{r=1}^{M}
 \frac{\|a(F_r)\|_{q_r}^{2}}{r^{2B}}
 \right]^{1/2}
 \le \Gamma_M(B)\|F\|_\infty
\end{equation}
for every analytic polynomial $F=\sum_{r=0}^{M}F_r$ in finitely many
variables.  Put
\[
 \overline\Gamma_M(B):=\max\{1,\Gamma_M(B)\}.
\]
The constants are finite and nondecreasing in $M$, since homogeneous
projections are contractive and
\[
 \Gamma_M(B)
 \le
 \left(\sum_{r=1}^{M}\frac{D_r^2}{r^{2B}}\right)^{1/2}.
\]

\begin{lemma}\label{lem:bootstrap-mixed-angle}
Let $0<a<1/2$.  If $u_i,v_i\ge0$ satisfy
$\sum_i u_i\le1$, $\sum_i v_i\le1$, and
$\theta_i\in[a,1-a]$, then
\begin{equation}\label{eq:bootstrap-mixed-angle}
 \sum_i u_i^{\theta_i}v_i^{1-\theta_i}
 \le
 2a^a(1-a)^{1-a}
 =2\exp\{-h(a)\},
\end{equation}
where
\[
 h(a):=-a\log a-(1-a)\log(1-a).
\]
\end{lemma}

\begin{proof}
Split the indices according to $u_i\ge v_i$ or $u_i<v_i$.  H\"older's
inequality gives
\[
 \sum_i u_i^{\theta_i}v_i^{1-\theta_i}
 \le
 U_+^{1-a}V_+^a+U_-^aV_-^{1-a},
\]
where $U_\pm,V_\pm$ are the corresponding partial masses.
The expression on the right is increasing in each partial mass.  Replace
$U_-$ by $1-U_+$ and $V_-$ by $1-V_+$.  These replacements can only increase
the expression, and $U_+\ge V_+$ implies $1-U_+\le1-V_+$.
Write $U_+=x$, $V_+=y$, with $0\le y\le x\le1$.  A second application of
H\"older gives
\begin{align*}
 x^{1-a}y^a+(1-x)^a(1-y)^{1-a}
 &\le (1+x-y)^{1-a}(1-x+y)^a.
\end{align*}
With $\delta=x-y\in[0,1]$, the right-hand side is
$(1+\delta)^{1-a}(1-\delta)^a$.  Its maximum occurs at
$\delta=1-2a$ and equals $2a^a(1-a)^{1-a}$.
\end{proof}

Let
\[
 Q(x,y)=\sum_{d,e\ge0}Q_{d,e}(x,y)
\]
be analytic of total degree at most $M$, with $Q_{d,e}$ bihomogeneous of
bidegree $(d,e)$.  Write
\[
 Q_{d,e}(x,y)
 =
 \sum_{|\beta|=d,\ |\gamma|=e}
 a^{d,e}_{\beta\gamma}x^\beta y^\gamma,
\]
and set
\begin{align*}
 X_{d,e}
 &:={}
 \left[
 \sum_{|\beta|=d}
 \left(\sum_{|\gamma|=e}|a^{d,e}_{\beta\gamma}|^2\right)^{q_d/2}
 \right]^{1/q_d},\\
 Y_{d,e}
 &:={}
 \left[
 \sum_{|\gamma|=e}
 \left(\sum_{|\beta|=d}|a^{d,e}_{\beta\gamma}|^2\right)^{q_e/2}
 \right]^{1/q_e},\\
 \widehat X_{d,e}
 &:={}
 \left(\frac d{d+1}\right)^{e/2}X_{d,e},
 \qquad
 \widehat Y_{d,e}
 :={}
 \left(\frac e{e+1}\right)^{d/2}Y_{d,e}.
\end{align*}

\begin{lemma}
\label{lem:bootstrap-normalized-row-column}
For every $B>0$,
\begin{align}
 \sum_{d,e\ge1}\frac{\widehat X_{d,e}^2}{d^{2B}}
 &\le
 \Gamma_M(B)^2\|Q\|_\infty^2,
 \label{eq:bootstrap-normalized-row}\\
 \sum_{d,e\ge1}\frac{\widehat Y_{d,e}^2}{e^{2B}}
 &\le
 \Gamma_M(B)^2\|Q\|_\infty^2.
 \label{eq:bootstrap-normalized-column}
\end{align}
\end{lemma}

\begin{proof}
By symmetry it suffices to prove the first estimate.  Fix $d$ and put
\[
 \lambda_d:=\sqrt{\frac d{d+1}}=\sqrt{\frac{q_d}{2}}.
\]
For $|\beta|=d$, let
\[
 B_{d,\beta}(y)
 :=
 \sum_{e\ge0}\sum_{|\gamma|=e}
 a^{d,e}_{\beta\gamma}y^\gamma.
\]
For $e\ge0$ and $|\beta|=d$, set
\[
 b_{e,\beta}:=
 \left(\sum_{|\gamma|=e}|a^{d,e}_{\beta\gamma}|^2\right)^{1/2}.
\]
Then
$\widehat X_{d,e}=\lambda_d^e
(\sum_{|\beta|=d}b_{e,\beta}^{q_d})^{1/q_d}$.  Since $q_d\le2$,
Minkowski's inequality, followed by the addition of the $e=0$ term and
Parseval in the $y$-variables, gives
\begin{align*}
 \left(\sum_{e\ge1}\widehat X_{d,e}^2\right)^{1/2}
 &\le
 \left[
 \sum_{|\beta|=d}
 \left(\sum_{e\ge1}\lambda_d^{2e}b_{e,\beta}^2\right)^{q_d/2}
 \right]^{1/q_d}\\
 &\le
 \left[
 \sum_{|\beta|=d}
 \|B_{d,\beta}(\lambda_d\,\cdot)\|_2^{q_d}
 \right]^{1/q_d}.
\end{align*}
Weissler's contraction \eqref{eq:weissler-full}, with $p=q_d$, yields
\[
 \|B_{d,\beta}(\lambda_d\,\cdot)\|_2
 \le \|B_{d,\beta}\|_{q_d}.
\]
If
\[
 g_d(y)
 :=
 \left[
 \sum_{|\beta|=d}|B_{d,\beta}(y)|^{q_d}
 \right]^{1/q_d},
\]
then Fubini gives
\[
 \left(\sum_{|\beta|=d}\|B_{d,\beta}\|_{q_d}^{q_d}\right)^{1/q_d}
 =\|g_d\|_{q_d}.
\]
Hence, since $q_d\le2$ and Haar measure is normalized,
\[
 \left(\sum_{e\ge1}\widehat X_{d,e}^2\right)^{1/2}
 \le \|g_d\|_{q_d}\le\|g_d\|_2.
\]
After multiplying by $d^{-B}$, squaring, and summing in $d$,
\[
 \sum_{d,e\ge1}\frac{\widehat X_{d,e}^2}{d^{2B}}
 \le
 \int_{\mathbb T^{n_y}}
 \sum_{d=1}^{M}\frac{g_d(y)^2}{d^{2B}}\,d\mathfrak m_{n_y}(y).
\]
For fixed $y$, $g_d(y)$ is exactly the coefficient $q_d$-norm of the
degree-$d$ component of $x\mapsto Q(x,y)$.  The definition of
$\Gamma_M(B)$ bounds the integrand by
$\Gamma_M(B)^2\|Q\|_\infty^2$.
\end{proof}

For $0<a<1/2$ and $r\ge1$, define
\begin{equation}\label{eq:bootstrap-central-set}
 \mathcal C_r(a)
 :=
 \left\{
 (d,e)\in\mathbb N^2:
 d+e=r,\quad ar\le d,e\le(1-a)r
 \right\}.
\end{equation}
For $R\ge10$, put
\[
 \kappa_R:=\sup_{r\ge R}(r+1)^{1/(2r)};
\]
then $\kappa_R\downarrow1$.

\begin{lemma}\label{lem:bootstrap-central}
For every $0<a<1/2$, $B>0$, $M\ge R$, and every analytic polynomial $Q$
of total degree at most $M$,
\begin{align}
 &\left[
 \sum_{r=R}^{M}\frac1{r^{2B}}
 \left(
 \sum_{(d,e)\in\mathcal C_r(a)}
 \|a(Q_{d,e})\|_{q_r}^{q_r}
 \right)^{2/q_r}
 \right]^{1/2}
 \notag\\
 &\hspace{18mm}\le
 \kappa_R\sqrt{2\e}\,
 \exp\{-(B+1/2)h(a)\}
 \Gamma_M(B)\|Q\|_\infty.
 \label{eq:bootstrap-central}
\end{align}
\end{lemma}

\begin{proof}
Fix $(d,e)\in\mathcal C_r(a)$, put $r=d+e$ and $\vartheta=d/r$.
By \Cref{lem:Blei},
\[
 \|a(Q_{d,e})\|_{q_r}
 \le X_{d,e}^{\vartheta}Y_{d,e}^{1-\vartheta}.
\]

The definitions of $\widehat X_{d,e}$ and $\widehat Y_{d,e}$ give the
multiplicative factor
\[
 A_{d,e}:=
 \left(1+\frac1d\right)^{de/(2r)}
 \left(1+\frac1e\right)^{de/(2r)}.
\]
Since
\[
 \log A_{d,e}
 =\frac{de}{2r}
 \left[\log\left(1+\frac1d\right)+\log\left(1+\frac1e\right)\right]
 \le\frac12,
\]
we have $A_{d,e}\le\e^{1/2}$. Thus, including when $Q_{d,e}=0$,
\[
 X_{d,e}^{\vartheta}Y_{d,e}^{1-\vartheta}
 \le
 \e^{1/2}
 \widehat X_{d,e}^{\vartheta}
 \widehat Y_{d,e}^{1-\vartheta},
\]
while
\[
 \frac{d^{B\vartheta}e^{B(1-\vartheta)}}{r^B}
 =\exp\{-Bh(\vartheta)\}
 \le\exp\{-Bh(a)\}.
\]
Thus, with
\[
 U_{d,e}:=\frac{\widehat X_{d,e}}{d^B},
 \qquad
 V_{d,e}:=\frac{\widehat Y_{d,e}}{e^B},
\]
we have
\[
 \left(\frac{\|a(Q_{d,e})\|_{q_r}}{r^B}\right)^2
 \le
 \e\,\exp\{-2Bh(a)\}
 U_{d,e}^{2\vartheta}V_{d,e}^{2(1-\vartheta)}.
\]
By \Cref{lem:bootstrap-normalized-row-column}, both square sums of
$U_{d,e}$ and $V_{d,e}$ are at most
$G^2:=\Gamma_M(B)^2\|Q\|_\infty^2$.
{If $G=0$, the row and column estimates force every central coefficient
array to vanish, and the claim follows. We may therefore assume $G>0$.}
Applying \Cref{lem:bootstrap-mixed-angle} to
$U_{d,e}^2/G^2$ and $V_{d,e}^2/G^2$ gives
\[
 \sum_{r=R}^{M}\sum_{(d,e)\in\mathcal C_r(a)}
 \left(\frac{\|a(Q_{d,e})\|_{q_r}}{r^B}\right)^2
 \le
 2\e\,\exp\{-(2B+1)h(a)\}
 \Gamma_M(B)^2\|Q\|_\infty^2.
\]
Finally, since $1/q_r-1/2=1/(2r)$ and
$\#\mathcal C_r(a)\le r+1$,
\[
 \left(
 \sum_{(d,e)\in\mathcal C_r(a)}
 \|a(Q_{d,e})\|_{q_r}^{q_r}
 \right)^{1/q_r}
 \le
 (r+1)^{1/(2r)}
 \left(
 \sum_{(d,e)\in\mathcal C_r(a)}
 \|a(Q_{d,e})\|_{q_r}^{2}
 \right)^{1/2}.
\]
Squaring and summing in $r$ proves the lemma.
\end{proof}

\subsection{Diffuse and dominant profiles}

Fix $\mu\in(1/2,1)$.  An $r$-homogeneous monomial $z^\alpha$ is
called \emph{$\mu$-diffuse} if
\[
 \max_j\alpha_j\le\mu r,
\]
and \emph{$\mu$-dominant} otherwise.  Write
\[
 F_r=F_r^{\rm diff}+F_r^{\rm dom}
\]
for the corresponding decomposition into disjoint coefficient families.
Under whole-coordinate coloring, every coordinate is sent independently to
$x$ or $y$, with equal probability.  If
\[
 S_\alpha:=\sum_j\alpha_j\xi_j,
 \qquad
 \mathbb P(\xi_j=0)=\mathbb P(\xi_j=1)=\frac12,
\]
then $S_\alpha$ is the resulting $x$-degree.  Since the pair of child
exponent vectors determines their sum $\alpha$, distinct parent coefficients
never merge under this coloring.

\begin{lemma}
\label{lem:bootstrap-cap-recovery}
Let $r\ge1$, $0\le k<r/2$, and put $s=k+1$.  Let $F_{r,k}$ be the
coefficient layer consisting of monomials whose unique largest exponent is
$r-k$.  There are random contractive projections followed by an exact
norm-preserving compression producing $s$-homogeneous polynomials
$G_{\omega,c,k}$ such that
\begin{equation}\label{eq:bootstrap-G-norm}
 \|G_{\omega,c,k}\|_\infty\le\|F_r\|_\infty
\end{equation}
and
\begin{equation}\label{eq:bootstrap-cap-mass}
 \frac14\|a(F_{r,k})\|_{q_r}^{q_r}
 \le
 \mathbb E_\omega\sum_{c=1}^{s}
 \|a(G_{\omega,c,k})\|_{q_r}^{q_r}.
\end{equation}
The compression preserves the coefficient array exactly.
\end{lemma}

\begin{proof}
Color the coordinates independently and uniformly by
$\{1,\ldots,s\}$.  For a coloring $\omega$ and a color $c$, put
\[
 J_c(\omega):=\{j:\omega(j)=c\}.
\]
The projection onto total color-$c$ degree $r-k$ is
\begin{equation}\label{eq:abstract-color-projection}
 (\mathcal P_{\omega,c,k}H)(z)
 :=
 \int_{\mathbb T}
 H\bigl((t^{\mathbf 1_{J_c(\omega)}(j)}z_j)_j\bigr)
 \overline t^{\,r-k}\,d\mathfrak m_1(t).
\end{equation}
Next average independently in every coordinate of color $c$ over the
$(r-k)$th roots of unity:
\begin{equation}\label{eq:abstract-root-projection}
 (\mathcal R_{\omega,c,k}H)(z)
 :=
 \frac1{(r-k)^{|J_c(\omega)|}}
 \sum_{\substack{\zeta_j^{\,r-k}=1\\j\in J_c(\omega)}}H(w),
\end{equation}
where $w_j=\zeta_jz_j$ for $j\in J_c(\omega)$ and $w_j=z_j$ otherwise.
Both operators are averages of isometries and hence contractive.

Set
\[
 R_{\omega,c,k}
 :=\mathcal R_{\omega,c,k}\mathcal P_{\omega,c,k}F_r.
\]
The first projection fixes the total color-$c$ degree at $r-k$; the second
forces each exponent on a color-$c$ coordinate to be a multiple of $r-k$.
Since $r-k>r/2$, exactly one such exponent is nonzero, and it equals $r-k$.
Thus
\[
 R_{\omega,c,k}(z)
 =
 \sum_{j:\omega(j)=c}
 z_j^{r-k}Q_{\omega,c,k,j}(z_{\omega\ne c}),
\]
where each $Q_{\omega,c,k,j}$ is $k$-homogeneous and depends only on
coordinates outside $J_c(\omega)$.  Replacing $z_j^{r-k}$ by a fresh variable
$x_j$ gives
\[
 G_{\omega,c,k}
 =
 \sum_{j:\omega(j)=c}
 x_jQ_{\omega,c,k,j}.
\]
For fixed $\omega,c,k$, the surviving monomials have exactly one color-$c$
coordinate with exponent $r-k$ and no other color-$c$ dependence.  Hence the
variables $z_j^{r-k}$ may be replaced independently by fresh variables $x_j$.
Indeed, the coordinatewise map
\[
 (z_j)_{j\in J_c(\omega)}\longmapsto
 (z_j^{\,r-k})_{j\in J_c(\omega)}
\]
is onto $\mathbb D^{J_c(\omega)}$.  Because the polynomials
$Q_{\omega,c,k,j}$ have no dependence on the color-$c$ coordinates, the
replacement therefore preserves both the coefficient array and the supremum
norm exactly.  Contractivity of
\eqref{eq:abstract-color-projection}--\eqref{eq:abstract-root-projection}
then proves \eqref{eq:bootstrap-G-norm}.

A parent monomial has residual degree $k=s-1$, hence at most $s-1$
residual active coordinates.  Conditional on the color of its dominant
coordinate, the probability that none receives that color is at least
$(1-1/s)^{s-1}\ge\e^{-1}>1/4$; for $s=1$ the probability is one.
On this event the coefficient is reproduced with modulus one, and distinct
parents remain distinct: the fresh variable records the dominant coordinate,
while the residual multi-index records all remaining exponents.  Averaging
proves \eqref{eq:bootstrap-cap-mass}.
\end{proof}

\begin{proposition}\label{prop:bootstrap-criterion}
Fix
\[
 0<a<\frac12,
 \qquad
 \frac12<\mu<1,
 \qquad
 0<p\le1,
 \qquad
 B>0,
\]
and suppose that every $\mu$-diffuse degree profile satisfies
\begin{equation}\label{eq:abstract-diffuse-capture}
 \mathbb P\{ar\le S_\alpha\le(1-a)r\}\ge p.
\end{equation}
Assume further that
\begin{equation}\label{eq:abstract-diffuse-contraction}
 p^{-1/2}\sqrt{2\e}\,
 \exp\{-(B+1/2)h(a)\}<1
\end{equation}
and that there exists $\theta\in(1-\mu,1)$ such that
\begin{equation}\label{eq:abstract-dominant-summability}
 B(1-\theta)>\frac32.
\end{equation}
Then
\begin{equation}\label{eq:abstract-uniform-Gamma}
 \sup_{M\ge1}\Gamma_M(B)<\infty.
\end{equation}
Consequently, $D_m\le K_Bm^B$ for some finite constant $K_B$.
\end{proposition}

\begin{proof}
For a whole-coordinate coloring $\omega$, let
\[
 \Phi_\omega(x,y)_j:=
 \begin{cases}
 x_j,&\xi_j(\omega)=1,\\
 y_j,&\xi_j(\omega)=0,
 \end{cases}
 \qquad
 Q_\omega(x,y):=F(\Phi_\omega(x,y)).
\]
Then $\|Q_\omega\|_\infty\le\|F\|_\infty$, and distinct monomials of $F$
produce distinct monomials of $Q_\omega$.

Choose $R$ so large that
\begin{equation}\label{eq:abstract-cR}
 c_R
 :=
 \kappa_Rp^{-1/q_R}\sqrt{2\e}\,
 \exp\{-(B+1/2)h(a)\}<1
\end{equation}
and
\[
 (1-\mu)r+1\le\theta r,
 \qquad \lfloor\theta r\rfloor\ge1
 \qquad(r\ge R).
\]

For the diffuse part, write $Q_{\omega,d,e}$ for the bihomogeneous
component of $Q_\omega$ of bidegree $(d,e)$ and set
\[
 C_r(\omega):=
 \left(
 \sum_{(d,e)\in\mathcal C_r(a)}
 \|a(Q_{\omega,d,e})\|_{q_r}^{q_r}
 \right)^{1/q_r}.
\]
Because distinct parent monomials remain distinct under the coloring,
\eqref{eq:abstract-diffuse-capture} gives
\[
 p\|a(F_r^{\rm diff})\|_{q_r}^{q_r}
 \le \mathbb E_\omega C_r(\omega)^{q_r}.
\]
Since $q_r\le2$,
\[
 \|a(F_r^{\rm diff})\|_{q_r}
 \le
 p^{-1/q_r}\bigl(\mathbb E_\omega C_r(\omega)^2\bigr)^{1/2}.
\]
Multiply by $r^{-B}$, square, sum over $r\ge R$, and apply
\Cref{lem:bootstrap-central} to $Q_\omega$, using
$\|Q_\omega\|_\infty\le\|F\|_\infty$.  Since
$p^{-1/q_r}\le p^{-1/q_R}$, we obtain
\begin{equation}\label{eq:abstract-diffuse-estimate}
 \left[
 \sum_{r=R}^{M}
 \frac{\|a(F_r^{\rm diff})\|_{q_r}^{2}}{r^{2B}}
 \right]^{1/2}
 \le c_R\Gamma_M(B)\|F\|_\infty.
\end{equation}

For the dominant part, write the largest exponent as $r-k$ and set
$s=k+1$.  Then $s\le\theta r$ for $r\ge R$.  The identity
\[
 \frac1{q_r}
 =
 \frac sr\frac1{q_s}
 +\left(1-\frac sr\right)\frac12,
\]
and log-convexity of finite-dimensional $\ell_p$ norms yield
\begin{align*}
 \|a(G_{\omega,c,k})\|_{q_r}
 &\le
 \|a(G_{\omega,c,k})\|_{q_s}^{s/r}
 \|a(G_{\omega,c,k})\|_2^{1-s/r}\\
 &\le
 \bigl(\Gamma_s(B)s^B\|G_{\omega,c,k}\|_\infty\bigr)^{s/r}
 \|G_{\omega,c,k}\|_\infty^{1-s/r}\\
 &\le
 \bigl(\Gamma_s(B)s^B\bigr)^{s/r}\|F_r\|_\infty.
\end{align*}
Here the second line uses the definition of $\Gamma_s(B)$ and Parseval,
and the last one uses the contractive projection estimate for the layer
$G_{\omega,c,k}$.
Insert this into \eqref{eq:bootstrap-cap-mass} and sum over the integers
$s=k+1\le\theta r$.  Distinct values of $k$ correspond to distinct largest
exponents $r-k$; hence the coefficient layers $F_{r,k}$ are disjoint.  Therefore
\[
 \|a(F_r^{\rm dom})\|_{q_r}^{q_r}
 \le
 4\|F_r\|_\infty^{q_r}
 \sum_{s\le\theta r}
 s\bigl(\Gamma_s(B)s^B\bigr)^{sq_r/r}.
\]
Put
\[
 \mathcal G_M:=\overline\Gamma_{\lfloor\theta M\rfloor}(B).
\]
Since $s\le\theta r\le\theta M$,
\[
 \Gamma_s(B)^{sq_r/r}\le\mathcal G_M^{\theta q_r},
 \qquad
 s^{Bsq_r/r}\le r^{B\theta q_r}.
\]
Using $\sum_{s\le\theta r}s\le r^2$ and contractivity of the homogeneous
projection, we obtain
\begin{equation}\label{eq:abstract-cap-pointwise}
 \frac{\|a(F_r^{\rm dom})\|_{q_r}}{r^B}
 \le
 C r^{2/q_r-B(1-\theta)}
 \mathcal G_M^\theta\|F\|_\infty.
\end{equation}
The square sum converges because
\[
 \frac4{q_r}-2B(1-\theta)
 \longrightarrow
 2-2B(1-\theta)<-1.
\]
Hence
\begin{equation}\label{eq:abstract-cap-estimate}
 \left[
 \sum_{r=R}^{M}
 \frac{\|a(F_r^{\rm dom})\|_{q_r}^{2}}{r^{2B}}
 \right]^{1/2}
 \le
 C_{a,\mu,B,\theta}
 \overline\Gamma_{\lfloor\theta M\rfloor}(B)^\theta
 \|F\|_\infty.
\end{equation}

The finitely many degrees below $R$ contribute a constant $C_0$.  Combining
\eqref{eq:abstract-diffuse-estimate} and
\eqref{eq:abstract-cap-estimate}, then absorbing the diffuse term, gives
\begin{equation}\label{eq:abstract-final-recurrence}
 \Gamma_M(B)
 \le
 A+E\overline\Gamma_{\lfloor\theta M\rfloor}(B)^\theta
\end{equation}
with finite constants $A,E$.  After enlarging the initial range if necessary,
$\lfloor\theta M\rfloor<M$ for every index to which the recurrence is applied.
Since $0<\theta<1$, choose $K\ge1$ larger
than the finitely many initial values and so large that $A+EK^\theta\le K$.
Strong induction in $M$ proves \eqref{eq:abstract-uniform-Gamma}.  Applying
\eqref{eq:Gamma-bootstrap} to a single homogeneous level gives the final
assertion.
\end{proof}

\subsection{Polynomial growth}

The criterion now closes with elementary, fixed parameters. No
optimization is needed to pass from the quasipolynomial bound to
polynomial growth.

\begin{lemma}\label{lem:elementary-admissible}
There are explicit parameters for which the contraction--compression criterion holds for every $B>9/2$. Consequently,
\begin{equation}\label{eq:elementary-polynomial-family}
 \sup_M\Gamma_M(B)<\infty
 \qquad\text{for every }B>\frac92.
\end{equation}
\end{lemma}

\begin{proof}
Set
\begin{equation}\label{eq:bootstrap-elementary-parameters}
 a_0:=\frac1{10},
 \qquad
 \mu_0:=\frac{5001}{10000},
 \qquad
 p_0:=\frac{1399}{6400}.
\end{equation}
If $\max_j\alpha_j\le\mu_0r$, then
\[
 \mathbb ES_\alpha=\frac r2,
 \qquad
 \operatorname{Var}(S_\alpha)
 =\frac14\sum_j\alpha_j^2
 \le\frac{\mu_0r^2}{4}.
\]
The distance from $r/2$ to the complement of $[r/10,9r/10]$ is $2r/5$.
Therefore
\[
 \mathbb P\{S_\alpha\notin[r/10,9r/10]\}
 \le\frac{25\mu_0}{16}
 =\frac{5001}{6400},
\]
which gives capture probability at least $p_0$.

\begingroup
\color{black}
\hypersetup{linkcolor=black,citecolor=black,urlcolor=black}
Put $c:=p_0^{-1/2}\sqrt{2\e}\,\exp\{-5h(a_0)\}$.
Since $10h(1/10)=\log10+9\log(10/9)$, the elementary inequalities
\[
 \log(1+u)>\frac{u}{1+u}\quad(u>0),
 \qquad
 -\log(1-v)>v+\frac{v^2}{2}\quad(0<v<1)
\]
give
\begin{align*}
 \log c^2
 &=1-\log(5p_0)-9\log(10/9)\\
 &<1-\frac{119}{1399}-\frac{189}{200}
 =-\frac{8411}{279800}<0.
\end{align*}
In particular,
\begin{equation}\label{eq:elementary-B-nine-halves}
 p_0^{-1/2}\sqrt{2\e}\,
 \exp\{-5h(a_0)\}<1.
\end{equation}
\endgroup
Thus \eqref{eq:abstract-diffuse-contraction} holds for every
$B\ge9/2$.  Since $1-\mu_0<1/2$, we may take $\theta=1/2$, and
\eqref{eq:abstract-dominant-summability} holds whenever $B>3$.
Applying \Cref{prop:bootstrap-criterion} proves \eqref{eq:elementary-polynomial-family}.
\end{proof}

\begin{proposition}\label{thm:main}
There are absolute constants $B_0<\infty$ and $K<\infty$ such that, for every
analytic polynomial
\[
 F=\sum_{r=0}^{M}F_r,
\]
where $F_r$ is $r$-homogeneous,
\begin{equation}\label{eq:weighted-polynomial-main}
 \left[
 \sum_{r=1}^{M}
 \frac{\|a(F_r)\|_{q_r}^{2}}{r^{2B_0}}
 \right]^{1/2}
 \le K\|F\|_\infty.
\end{equation}
Consequently,
\begin{equation}\label{eq:polynomial-main}
 D_m\le K m^{B_0}
 \qquad(m\ge1).
\end{equation}
One may take $B_0=5$.
\end{proposition}

\begin{proof}
Take $B_0=5$ in \eqref{eq:elementary-polynomial-family}.  The weighted
estimate \eqref{eq:weighted-polynomial-main} is exactly
\eqref{eq:Gamma-bootstrap} with a uniform bound for $\Gamma_M(5)$; applying
it to one homogeneous level gives \eqref{eq:polynomial-main}.
\end{proof}

\begingroup
\color{black}
\hypersetup{linkcolor=black,citecolor=black,urlcolor=black}
\section{Proof of Theorem A: polynomial Bohnenblust--Hille growth}
\label{sec:proof-A}

\begin{proof}
By \Cref{thm:main}, the weighted estimate holds uniformly with $B_0=5$.
Apply it to one $m$-homogeneous polynomial and take the supremum over all
finite dimensions. This gives $D_m\le K m^5$, and proves the theorem.
\end{proof}

\begingroup
\definecolor{remarkpurple}{RGB}{0,0,0}
\color{black}
\hypersetup{linkcolor=black,citecolor=black,urlcolor=black}
\begin{remark}[Square-root-scale estimate]
\label{rem:sqrt-scale-refinement}
\begingroup\color{black}
There are absolute constants $C,m_0<\infty$ such that
\[
 D_m\le \sqrt m\,
 \exp\!\bigl(C\sqrt{\log m}\,\log\log m\bigr)
 \qquad(m\ge m_0).
\]
A complete proof is given in Appendix~\ref{app:sqrt-scale-refinement}.

The proof uses a critical weight
\[
 \W(r)=\sqrt r\,
 \exp\!\bigl(\sqrt{\log r}\,\log\log r\bigr)
\]
(up to an immaterial modification in bounded degree), interpolation over
complements of $L$ coordinate blocks, an entropy form of the
arithmetic--geometric mean inequality, Bernstein concentration for the
random block degrees, and the same coefficient-preserving compression used
in the polynomial-growth argument.  The block number is chosen on the scale
\[
 L\asymp \frac{\sqrt{\log M}}{\log\log M},
\]
which yields a diffuse contraction of order
$\exp(-\gamma/L^2)$ and closes the graded bootstrap with
\[
 \log\Gamma_M\le C\sqrt{\log M}\,\log\log M.
\]
\endgroup
\end{remark}
\endgroup

\endgroup

\section{Proof of Theorem B: the critical diagonal}\label{sec:critical-limit-two}

Let $\mathcal P_{\le m}(\C^n)$ denote the polynomials of total degree at
most $m$.  For
\[
 Q(z)=\sum_{|\alpha|\le m}a_\alpha z^\alpha
\]
define its $m$-homogenization by
\begin{equation}\label{eq:critical-homogenization-map}
 \mathcal H_{m,n}Q(z_0,z)
 :=\sum_{|\alpha|\le m}a_\alpha z_0^{m-|\alpha|}z^\alpha.
\end{equation}
For $(z_0,z)\in\mathbb T^{n+1}$,
\[
 (\mathcal H_{m,n}Q)(z_0,z)
 =z_0^mQ(z_1/z_0,\ldots,z_n/z_0).
\]
Taking suprema on the distinguished boundary and then using the maximum
modulus principle in each variable gives
\begin{equation}\label{eq:critical-homogenization-norm}
 \|\mathcal H_{m,n}Q\|_\infty=\|Q\|_\infty.
\end{equation}
Moreover, the map $\alpha\mapsto(m-|\alpha|,\alpha)$ is injective, so
homogenization preserves every coefficient $\ell_q$-norm.

The proof has two independent components. The upper estimate for $D_{m,n}$ will follow at the end from the dimension of the space of $m$-homogeneous polynomials. The point of the next proposition is the reverse estimate. We first construct a bounded polynomial whose coefficient mass is concentrated in degrees $d+o(m)$, and then homogenize it. The Blaschke factor determines the coefficient mass, whereas the Fej\'er kernel performs the degree localization.

\begin{proposition}\label{prop:critical-finite-sharpness}
Suppose $n_m/m\to1$.  There exists a sequence $(P_m)$, defined for all
sufficiently large $m$, where each $P_m$ is an $m$-homogeneous polynomial on
$\C^{n_m}$, such that
\[
 \|P_m\|_\infty\le1,
 \qquad
 \|P_m\|_{\ell_{q_m}}\longrightarrow2,
 \qquad q_m:=\frac{2m}{m+1}.
\]
\end{proposition}
\begin{proof}
Put
\[
 h:=\lceil m^{3/4}\rceil,
 \qquad
 d:=\min\{n_m-1,m-h\}.
\]
Then $d/m\to1$, $d+h\le m$, and $d+1\le n_m$ for all sufficiently large
$m$.  Consider the Blaschke factor
\[
 b(z):=\frac{2^{-1/2}-z}{1-2^{-1/2}z}
      =\sum_{j=0}^\infty c_jz^j,
 \qquad
 c_0=2^{-1/2},
 \quad
 c_j=-2^{-(j+1)/2}\quad(j\ge1).
\]
It satisfies $|b(z)|\le1$ on $\overline{\mathbb D}$ and $|b(z)|=1$ on
$\mathbb T$.  Hence, for
\[
 B_d(z):=\prod_{\ell=1}^d b(z_\ell)
        =\sum_{\alpha\in\mathbb N_0^d}c_\alpha z^\alpha,
 \qquad c_\alpha:=\prod_{\ell=1}^dc_{\alpha_\ell},
\]
we have $\|B_d\|_\infty=1$ and
\begin{equation}\label{eq:critical-coefficient-square}
 |c_\alpha|^2=2^{-(d+|\alpha|)}.
\end{equation}

For $j\in\mathbb Z$, set
\[
 w_h(j):=\left(1-\frac{|j|}{h+1}\right)_+,
 \qquad
 F_h(t):=\frac1{h+1}
 \left|\sum_{\ell=0}^h e^{i\ell t}\right|^2.
\]
The Fej\'er expansion is
\[
 F_h(t)=\sum_{j=-h}^h w_h(j)e^{ijt},
 \qquad
 \int_0^{2\pi}F_h(t)\,\frac{dt}{2\pi}=1.
\]
Define
\[
 Q_m(z):=
 \sum_{\substack{\alpha\in\mathbb N_0^d\\|\alpha|\le d+h}}
 w_h(|\alpha|-d)c_\alpha z^\alpha.
\]
Termwise integration gives
\[
 Q_m(z)=\int_0^{2\pi}
 F_h(t)e^{-idt}B_d(e^{it}z)\,\frac{dt}{2\pi}.
\]
Since $F_h\ge0$, it follows that $\|Q_m\|_\infty\le1$.  As
$\deg Q_m\le d+h\le m$, define $P_m:=\mathcal H_{m,d}Q_m$ and regard it
as a polynomial on $\C^{n_m}$ by ignoring the remaining variables.  By
\eqref{eq:critical-homogenization-norm}, $\|P_m\|_\infty\le1$.

It remains to compute its coefficient norm.  Let $\nu$ be the probability
measure on $\mathbb N_0$ given by
\[
 \nu(\{j\})=2^{-(j+1)},
\]
let $\nu_d:=\nu^{\otimes d}$, and put
$S(\alpha):=|\alpha|$.  Then
\[
 \int S\,d\nu_d=d,
 \qquad
 \int(S-d)^2\,d\nu_d=2d.
\]
Since $1-q_m/2=1/(m+1)$, \eqref{eq:critical-coefficient-square} yields
\begin{equation}\label{eq:critical-expectation-BH}
 \|P_m\|_{\ell_{q_m}}^{q_m}
 =\int_{\mathbb N_0^d}
 w_h(S-d)^{q_m}2^{(d+S)/(m+1)}\,d\nu_d.
\end{equation}
Thus \eqref{eq:critical-expectation-BH} converts the coefficient norm into an expectation for a sum of $d$ independent variables of mean $1$ and variance $2$. Since its standard deviation is of order $\sqrt d$, the choice $h\gg\sqrt m$ ensures that the Fej\'er weight is asymptotically $1$ on a set of probability tending to $1$.
Chebyshev's inequality gives
\[
 \nu_d\{|S-d|>m^{5/8}\}
 \le \frac{2d}{m^{5/4}}
 \le2m^{-1/4}.
\]
On $|S-d|\le m^{5/8}$,
\[
 w_h(S-d)\ge1-m^{-1/8},
 \qquad
 d+S\ge2d-m^{5/8}.
\]
Consequently,
\[
 (1-2m^{-1/4})(1-m^{-1/8})^{q_m}
 2^{(2d-m^{5/8})/(m+1)}
 \le \|P_m\|_{\ell_{q_m}}^{q_m}.
\]
For the reverse estimate, since $0\le w_h\le1$,
\begin{align*}
 \|P_m\|_{\ell_{q_m}}^{q_m}
 &\le
 2^{d/(m+1)}
 \left(\sum_{j=0}^\infty 2^{-(j+1)}2^{j/(m+1)}\right)^d.
\end{align*}
Set $x_m:=2^{1/(m+1)}$.  Then the right-hand side equals
\[
 \left(\frac{x_m}{2-x_m}\right)^d.
\]
Since $x_m=1+(\log2)/(m+1)+O(m^{-2})$,
\[
 \log\frac{x_m}{2-x_m}
 =\frac{2\log2}{m+1}+O(m^{-2}),
\]
and therefore, because $d/m\to1$,
\[
 d\log\frac{x_m}{2-x_m}\longrightarrow2\log2.
\]
Thus $\limsup_m\|P_m\|_{\ell_{q_m}}^{q_m}\le4$.  The lower bound above has limit $4$, so
\[
 \|P_m\|_{\ell_{q_m}}^{q_m}\longrightarrow4,
 \qquad
 \|P_m\|_{\ell_{q_m}}\longrightarrow2.
\]
\end{proof}

\begin{proof}[Proof of Theorem~\ref{thm:B}]
The number of degree-$m$ monomials in $n$ variables is
$N_{m,n}=\binom{m+n-1}{m}$.  Parseval's identity and the elementary inclusion
$\ell_2^{N_{m,n}}\hookrightarrow\ell_{q_m}^{N_{m,n}}$ give
\begin{equation}\label{eq:critical-upper-BH}
 D_{m,n}\le N_{m,n}^{1/(2m)}
 =\binom{m+n-1}{m}^{1/(2m)}.
\end{equation}
If $n_m/m\to1$, Stirling's formula gives
\[
 \frac1m\log\binom{m+n_m-1}{m}\longrightarrow2\log2,
\]
and hence \eqref{eq:critical-upper-BH} implies
$\limsup_mD_{m,n_m}\le2$.  Proposition~\ref{prop:critical-finite-sharpness}
gives the reverse inequality:
\[
 D_{m,n_m}
 \ge\frac{\|P_m\|_{\ell_{q_m}}}{\|P_m\|_\infty}
 \ge\|P_m\|_{\ell_{q_m}}\longrightarrow2.
\]
Thus $D_{m,n_m}\to2$.  Taking $n_m=m$ and using $D_m\ge D_{m,m}$ gives
$\liminf_{m\to\infty}D_m\ge2$.
\end{proof}

\part{Sidon constants and the multidimensional Bohr radius}

\section{Proof of Theorem C: homogeneous Sidon constants}
\label{sec:sidon}

We first prove the lower estimate in Theorem~\ref{thm:C} by the unitary-chain
construction.  The upper estimate then follows from Theorem~\ref{thm:A} and
H\"older's inequality.

The lower estimate is based on quadratic and quartic averages of Haar unitary
matrix entries.  The combinatorial point in the quartic calculation is isolated first.

\begin{lemma}[Binary pairing-chain bound]\label{lem:binary-pairing-chain}
Let $m\ge2$, let $a,b\ge0$, and let $\mathfrak S_m$ denote the symmetric group on $\{1,\ldots,m\}$.  For $\sigma,\tau\in\mathfrak S_m$, put
\[
 f(\sigma,\tau):=\card\{1\le r\le m:\sigma(r)=\tau(r)\}.
\]
For $1\le r\le m-1$, choose bits $L_r,R_r\in\{0,1\}$ and impose the
compatibility condition
\[
 R_{r-1}=L_r
 \qquad
 \text{whenever }2\le r\le m-1
 \text{ and }\sigma(r)\ne\tau(r).
\]
Set $w(u,v)=a$ when $u=v$ and $w(u,v)=b$ when $u\ne v$.  Then
\begin{equation}\label{eq:binary-pairing-chain}
 \sum_{(L,R)}\prod_{r=1}^{m-1}w(L_r,R_r)
 \le
 2^{f(\sigma,\tau)+1}(a+b)^{m-1},
\end{equation}
where the sum is over all compatible choices of the bits.
\end{lemma}

\begin{proof}
Delete the vertices $r$ for which $\sigma(r)=\tau(r)$.  The remaining
compatibility conditions split the edge set $\{1,\ldots,m-1\}$ into at
most $f(\sigma,\tau)+1$ consecutive components.  On a component containing
$\ell$ edges, the sum over its binary states is
\[
 (1,1)
 \begin{pmatrix}a&b\\ b&a\end{pmatrix}^{\ell}
 \binom11
 =2(a+b)^{\ell}.
\]
Multiplying over the components gives
$2^c(a+b)^{m-1}$ with $c\le f(\sigma,\tau)+1$, which proves
\eqref{eq:binary-pairing-chain}.
\end{proof}

\begin{proposition}[Unitary-chain Sidon lower bound]\label{prop:sidon-lower}
Let $m\ge2$ and $n\ge m^2$. Then
\begin{equation}\label{eq:sidon-lower-exact}
 \Sid_{m,n}
 \ge
 \frac{1}{2\sqrt{\e}}\,
 \frac{\sqrt{m!}\binom{n}{m}}{n^{(m+1)/2}}.
\end{equation}
Consequently,
\begin{equation}\label{eq:sidon-lower-classical}
 \Sid_{m,n}
 \ge
 \frac{1}{4\sqrt{\e m}}\,
 \sqrt{\frac{n^{m-1}}{(m-1)!}}.
\end{equation}
\end{proposition}

\begin{proof}
Let $\mathrm U(n)$ denote the unitary group on $\mathbb C^n$, and let
$U_1,\ldots,U_{m-1}$ be independent Haar-distributed matrices in $\mathrm U(n)$.
All expectations below are taken with respect to these matrices.  Put
\[
 A_r:=\sqrt n\,U_r
 \qquad(1\le r\le m-1).
\]
For $z=(z_1,\ldots,z_n)\in\C^n$, let
\[
 D_z:=\operatorname{diag}(z_1,\ldots,z_n)
\]
and define
\begin{equation}\label{eq:unitary-chain-polynomial}
 P(z):=
 z^{\mathsf T}A_1D_zA_2D_z\cdots D_zA_{m-1}z.
\end{equation}
Equivalently,
\[
 P(z)=
 \sum_{i_1,\ldots,i_m=1}^n
 \left(\prod_{r=1}^{m-1}(A_r)_{i_r i_{r+1}}\right)
 z_{i_1}\cdots z_{i_m}.
\]
Since
\[
 \|A_r\|_{2\to2}=\sqrt n,
 \qquad
 \|D_z\|_{2\to2}\le1,
 \qquad
 \|z\|_2\le\sqrt n
 \quad(z\in\D^n),
\]
we have, for every realization of the matrices,
\begin{equation}\label{eq:unitary-chain-supnorm}
 \|P\|_\infty\le n^{(m+1)/2}.
\end{equation}

Fix a set
\[
 J=\{j_1,\ldots,j_m\}\subset\{1,\ldots,n\}
\]
of cardinality $m$.  The coefficient of the square-free monomial
$z_{j_1}\cdots z_{j_m}$ in $P$ is
\begin{equation}\label{eq:unitary-chain-squarefree-coefficient}
 c_J=
 \sum_{\sigma\in\mathfrak S_m}
 \prod_{r=1}^{m-1}
 (A_r)_{j_{\sigma(r)},j_{\sigma(r+1)}}.
\end{equation}
For a Haar matrix $U\in\mathrm U(n)$ and $A=\sqrt n\,U$,
\[
 \mathbb E\,A_{ij}\overline{A_{kl}}=\delta_{ik}\delta_{jl}.
\]
Independence of $A_1,\ldots,A_{m-1}$ therefore gives
\begin{equation}\label{eq:unitary-chain-second-moment}
 \mathbb E|c_J|^2=m!.
\end{equation}

We next estimate the quartic average.  
The unitary Weingarten formula
\cite[Corollary~2.4]{CollinsSniady2006}, specialized to order two, gives
\begin{equation}\label{eq:haar-four-entry}
\begin{split}
 &\mathbb E\!
 \left[
 A_{i_1j_1}A_{i_2j_2}
 \overline{A_{i'_1j'_1}A_{i'_2j'_2}}
 \right]\\
 &\quad=
 a_n\bigl(
 \delta_{i_1i'_1}\delta_{i_2i'_2}
 \delta_{j_1j'_1}\delta_{j_2j'_2}
 +
 \delta_{i_1i'_2}\delta_{i_2i'_1}
 \delta_{j_1j'_2}\delta_{j_2j'_1}
 \bigr)\\
 &\qquad\ -
 b_n\bigl(
 \delta_{i_1i'_1}\delta_{i_2i'_2}
 \delta_{j_1j'_2}\delta_{j_2j'_1}
 +
 \delta_{i_1i'_2}\delta_{i_2i'_1}
 \delta_{j_1j'_1}\delta_{j_2j'_2}
 \bigr),
\end{split}
\end{equation}
where
\[
 a_n:=\frac{n^2}{n^2-1},
 \qquad
 b_n:=\frac{n}{n^2-1}.
\]
In particular,
\begin{equation}\label{eq:haar-weight-sum}
 a_n+b_n=\frac{n}{n-1}.
\end{equation}

Expand $\mathbb E|c_J|^4$ using
\eqref{eq:unitary-chain-squarefree-coefficient} and apply
\eqref{eq:haar-four-entry} independently to each matrix $A_r$.  Fix the two
nonconjugated permutations $\sigma,\tau\in\mathfrak S_m$.  At edge $r$,
encode the row pairing in \eqref{eq:haar-four-entry} by a bit $L_r$ and the
column pairing by a bit $R_r$: $0$ denotes the identity pairing and $1$ the
transposition.  After absolute values are taken, the weight of the edge is
$a_n$ if $L_r=R_r$ and $b_n$ otherwise.

At an internal vertex $r$, the column indices on edge $r-1$ are precisely the
row indices on edge $r$.  Hence the Kronecker conditions contributed by those
two edges must agree on the same ordered pair of labels.  If
$\sigma(r)\ne\tau(r)$, the two labels are distinct, so the identity and
transposition pairings are distinguished and necessarily
$R_{r-1}=L_r$.  If $\sigma(r)=\tau(r)$, both labels coincide, and the two
pairings impose the same equalities; in that case no compatibility condition
is forced.  Thus, for fixed $\sigma,\tau$, every nonzero term in the expanded
Weingarten sum determines a compatible binary chain in the sense of
Lemma~\ref{lem:binary-pairing-chain}.

Conversely, after $\sigma,\tau$ and a compatible chain are fixed, the
Kronecker conditions determine the two conjugated labels at each vertex.  A
term from the original expansion is present only when each of the two resulting
label strings is a permutation of $J$.  If we drop this permutation
requirement after taking absolute values, we only enlarge the sum.  Therefore
Lemma~\ref{lem:binary-pairing-chain}, with $a=a_n$ and $b=b_n$, bounds the
total contribution associated with the fixed pair $(\sigma,\tau)$ by
\begin{equation}\label{eq:fixed-pair-fourth-bound}
 2^{f(\sigma,\tau)+1}(a_n+b_n)^{m-1}.
\end{equation}
Therefore
\begin{align}
 \mathbb E|c_J|^4
 &\le
 2(a_n+b_n)^{m-1}
 \sum_{\sigma,\tau\in\mathfrak S_m}2^{f(\sigma,\tau)}\notag\\
 &=
 2(a_n+b_n)^{m-1}
 m!\sum_{\pi\in\mathfrak S_m}2^{\operatorname{fix}(\pi)}.
 \label{eq:fourth-before-fixed-points}
\end{align}
Here $\operatorname{Fix}(\pi):=\{j:\pi(j)=j\}$ and
$\operatorname{fix}(\pi):=\card\operatorname{Fix}(\pi)$.  The last combinatorial sum is explicit.  Indeed,
\[
 \sum_{\pi\in\mathfrak S_m}2^{\operatorname{fix}(\pi)}
 =
 \sum_{\pi\in\mathfrak S_m}
 \sum_{J\subset\operatorname{Fix}(\pi)}1
 =
 \sum_{k=0}^m\binom mk(m-k)!
 =
 m!\sum_{k=0}^m\frac1{k!}
 \le \e\,m!.
\]
Together with \eqref{eq:haar-weight-sum}, this yields
\begin{equation}\label{eq:unitary-chain-fourth-moment}
 \mathbb E|c_J|^4
 \le
 2\e\left(\frac{n}{n-1}\right)^{m-1}(m!)^2.
\end{equation}

Interpolation between $L^1$, $L^2$, and $L^4$ gives
\[
 (\mathbb E|c_J|^2)^{3/2}
 \le
 \mathbb E|c_J|\,
 (\mathbb E|c_J|^4)^{1/2}.
\]
Using \eqref{eq:unitary-chain-second-moment} and
\eqref{eq:unitary-chain-fourth-moment},
\begin{equation}\label{eq:unitary-chain-first-moment}
 \mathbb E|c_J|
 \ge
 \frac{1}{\sqrt{2\e}}
 \left(1-\frac1n\right)^{(m-1)/2}\sqrt{m!}.
\end{equation}
Since $n\ge m^2$ and $m\ge2$, Bernoulli's inequality gives
\[
 \left(1-\frac1n\right)^{m-1}
 \ge
 1-\frac{m-1}{n}
 \ge\frac12.
\]
Thus
\begin{equation}\label{eq:unitary-chain-first-moment-simple}
 \mathbb E|c_J|\ge\frac{1}{2\sqrt{\e}}\sqrt{m!}.
\end{equation}
Summing \eqref{eq:unitary-chain-first-moment-simple} over the
$\binom nm$ square-free monomials shows that some realization of the matrices
satisfies
\[
 \sum_{\substack{J\subset\{1,\ldots,n\}\\ \card J=m}}|c_J|
 \ge
 \frac{1}{2\sqrt{\e}}\sqrt{m!}\binom nm.
\]
The full coefficient $\ell_1$ norm of $P$ is at least this square-free mass.
Combining with \eqref{eq:unitary-chain-supnorm} proves
\eqref{eq:sidon-lower-exact}.

Finally,
\begin{equation}\label{eq:unitary-chain-scale-identity}
 \frac{\sqrt{m!}\binom nm}{n^{(m+1)/2}}
 =
 \frac1{\sqrt m}
 \sqrt{\frac{n^{m-1}}{(m-1)!}}
 \prod_{j=0}^{m-1}\left(1-\frac jn\right).
\end{equation}
For $n\ge m^2$,
\[
 \prod_{j=0}^{m-1}\left(1-\frac jn\right)
 \ge
 1-\sum_{j=0}^{m-1}\frac jn
 \ge\frac12.
\]
Equation \eqref{eq:sidon-lower-classical} follows.
\end{proof}

\begin{proof}[Proof of Theorem~\ref{thm:C}]
Let $B_0,K$ be as in Theorem~\ref{thm:A}. For $m,n\ge1$, let
\[
 P(z)=\sum_{|\alpha|=m}a_\alpha z^\alpha
 \qquad(z\in\C^n).
\]
There are
\[
 N_{m,n}:=\binom{n+m-1}{m}
\]
monomials of degree $m$ in $n$ variables. Since
\[
 1-\frac1{q_m}=\frac{m-1}{2m},
\]
H\"older's inequality and Theorem~\ref{thm:A} give, for every $m,n\ge1$,
\begin{equation}\label{eq:sidon-global-upper}
 \sum_{|\alpha|=m}|a_\alpha|
 \le
Km^{B_0} N_{m,n}^{(m-1)/(2m)}\|P\|_\infty.
\end{equation}
Assume now that $m\ge2$ and $n\ge m^2$. Then
\begin{align*}
 N_{m,n}
 &=\frac{n^m}{m!}\prod_{j=0}^{m-1}\left(1+\frac jn\right)\\
 &\le \frac{n^m}{m!}
 \exp\!\left(\frac{m(m-1)}{2n}\right)
 \le \e^{1/2}\frac{n^m}{m!}.
\end{align*}
Hence
\[
 N_{m,n}^{(m-1)/(2m)}
 \le \e^{1/4}n^{(m-1)/2}(m!)^{-(m-1)/(2m)}.
\]
Since $(m-1)!\le m^{m-1}$,
\[
 (m!)^{-(m-1)/(2m)}
 \le ((m-1)!)^{-1/2},
\]
and therefore
\[
 \Sid_{m,n}
 \le \e^{1/4}Km^{B_0}
 \sqrt{\frac{n^{m-1}}{(m-1)!}}.
\]
The opposite inequality in \eqref{eq:theoremC-sidon} is
Proposition~\ref{prop:sidon-lower}. This proves Theorem~\ref{thm:C}.
\end{proof}

\begin{corollary}[Global upper bound and Dirichlet form]\label{cor:sidon-dirichlet}
With $B_0,K$ as in Theorem~\ref{thm:A}, one has
\begin{equation}\label{eq:main-sidon-upper}
 \Sid_{m,n}
 \le Km^{B_0}
 \binom{n+m-1}{m}^{(m-1)/(2m)}
 \qquad(m,n\ge1).
\end{equation}
Here $\Omega(k)$ denotes the number of prime factors of $k$, counted with
multiplicity. If $A\subset\{k\in\N:\Omega(k)=m\}$ is finite, then
\begin{equation}\label{eq:main-dirichlet-sidon}
 \sum_{k\in A}|a_k|
 \le Km^{B_0}|A|^{(m-1)/(2m)}
 \sup_{t\in\R}\left|\sum_{k\in A}a_k k^{-it}\right|.
\end{equation}
\end{corollary}

\begin{proof}
The first assertion is exactly \eqref{eq:sidon-global-upper}. For the second,
let $p_1,\ldots,p_N$ be the primes occurring in the factorizations of elements
of $A$. For $k\in A$, write
\[
 k=p_1^{\alpha_1}\cdots p_N^{\alpha_N},
 \qquad |\alpha|=m,
\]
and associate
\[
 F(z):=\sum_{k\in A}a_k z^\alpha,
 \qquad z\in\C^N.
\]
The Bohr correspondence and Kronecker's theorem give
\[
 \|F\|_\infty
 =\sup_{t\in\R}\left|\sum_{k\in A}a_k k^{-it}\right|.
\]
Applying H\"older's inequality on the actual support and then
Theorem~\ref{thm:A},
\[
 \sum_{k\in A}|a_k|
 \le |A|^{(m-1)/(2m)}
 \left(\sum_{k\in A}|a_k|^{q_m}\right)^{1/q_m}
 \le Km^{B_0}|A|^{(m-1)/(2m)}\|F\|_\infty.
\]
This proves \eqref{eq:main-dirichlet-sidon}.
\end{proof}

\section{Preparatory estimates for Theorem D}\label{sec:prelim-D}

Recall that $\mathfrak b_n$ denotes the Bohr radius of the polydisc $\D^n$, namely the
largest $r\in(0,1)$ such that every polynomial
\[
 f(z)=\sum_{\alpha\in\N_0^n}a_\alpha z^\alpha
\]
satisfies
\begin{equation}\label{eq:bohr-definition}
 \sum_\alpha |a_\alpha|r^{|\alpha|}
 \le \|f\|_{L^\infty(\D^n)}.
\end{equation}
Theorem~\ref{thm:A} enters the Bohr-radius problem through the following estimate.

\begin{lemma}\label{lem:bohr-saddle}
For every $\beta\ge0$ there is $C_\beta<\infty$ such that
\begin{equation}\label{eq:bohr-saddle}
 \sum_{m=1}^\infty
 m^\beta\left(\frac{\e x}{m}\right)^{m/2}
 \le C_\beta \e^{x/2}x^{\beta+1/2}
 \qquad(x\ge2).
\end{equation}
\end{lemma}

\begin{proof}
Set
\[
 \phi_x(t):=\frac t2\left(1+\log x-\log t\right),
 \qquad t>0.
\]
Then $\phi_x$ is strictly concave, with unique maximum
$\phi_x(x)=x/2$.  On $[x/2,2x]$ one has
$\phi_x''(t)=-1/(2t)\le-1/(4x)$, and therefore
\[
 \phi_x(t)\le \frac x2-\frac{(t-x)^2}{8x}.
\]
It follows that the contribution of the integers in this interval is at most
\[
 C_\beta x^\beta\e^{x/2}
 \sum_{m\in\mathbb Z}\exp\!\left(-\frac{(m-x)^2}{8x}\right)
 \le C_\beta\e^{x/2}x^{\beta+1/2}.
\]

For $t\le x/2$, monotonicity gives
\[
 \phi_x(t)\le\phi_x(x/2)
 =\frac x2-\frac{1-\log2}{4}x.
\]
There are at most $x/2$ such integers, so their total contribution is
$O_\beta(x^{\beta+1}\e^{x/2-cx})$ for some $c>0$.

For $t\ge2x$, one has $\phi_x'(t)\le-(\log2)/2$.  Thus
\[
 \phi_x(m)
 \le \phi_x(2x)-\frac{\log2}{2}(m-2x),
 \qquad m\ge2x,
\]
and the polynomially weighted geometric series is bounded by
$C_\beta x^\beta\e^{\phi_x(2x)}$.  Since
$\phi_x(2x)=x(1-\log2)<x/2$, both tails are exponentially smaller than the
right-hand side of \eqref{eq:bohr-saddle}.
\end{proof}

\begin{proposition}
\label{prop:polynomial-to-bohr}
Suppose that
\[
 D_m\le C_0m^B\qquad(m\ge1)
\]
for some $B\ge0$.  Then, for every
\begin{equation}\label{eq:A-range}
 A>2B+\frac32,
\end{equation}
one has
\begin{equation}\label{eq:bohr-generic-lower}
 \mathfrak b_n
 \ge
 \sqrt{\frac{\log n-A\log\log n}{n}}
\end{equation}
for all sufficiently large $n$.
\end{proposition}

\begin{proof}
Write $L=\log n$ and
\[
 x:=L-A\log L,
 \qquad
 r:=\sqrt{\frac{x}{n}}.
\]
For large $n$, one has $x\ge2$.  Let
$f=\sum_{m\ge0}P_m$ be a polynomial on $\C^n$ with $\|f\|_\infty\le1$, and
write $a_0=P_0$.  Wiener's lemma in its polydisc form gives
\begin{equation}\label{eq:wiener-homogeneous}
 \|P_m\|_\infty\le1-|a_0|^2
 \qquad(m\ge1);
\end{equation}
see~\cite[Lemma~6.1]{BayartPellegrinoSeoane2014}.  H\"older's inequality and the assumed bound for
$D_m$
therefore yield
\begin{equation}\label{eq:bohr-majorant}
 \sum_\alpha|a_\alpha|r^{|\alpha|}
 \le |a_0|+(1-|a_0|^2)\Sigma_n(r),
\end{equation}
where
\begin{equation}\label{eq:Sigma-def}
 \Sigma_n(r)
 :=C_0\sum_{m\ge1}
 m^B r^m
 \binom{n+m-1}{m}^{(m-1)/(2m)}.
\end{equation}
Suppose first that $m\le\sqrt n$.  Since
\[
 \binom{n+m-1}{m}
 =\frac{n^m}{m!}\prod_{j=0}^{m-1}\left(1+\frac jn\right)
 \le\frac{n^m}{m!}
 \exp\!\left(\frac{m(m-1)}{2n}\right),
\]
we obtain
\begin{align*}
 r^m\binom{n+m-1}{m}^{(m-1)/(2m)}
 &\le
 \e^{1/4}n^{-1/2}x^{m/2}(m!)^{-(m-1)/(2m)}.
\end{align*}
Stirling's lower bound $m!\ge c\sqrt m\,(m/\e)^m$ gives
\[
 (m!)^{-(m-1)/(2m)}
 \le C m^{1/4}\left(\frac{\e}{m}\right)^{m/2}.
\]
Consequently,
\begin{equation}\label{eq:bohr-small-degree-term}
 r^m\binom{n+m-1}{m}^{(m-1)/(2m)}
 \le
 Cn^{-1/2}m^{1/4}
 \left(\frac{\e x}{m}\right)^{m/2}.
\end{equation}
Applying \Cref{lem:bohr-saddle} with $\beta=B+1/4$ yields
\begin{align}
 \sum_{m\le\sqrt n}
 m^B r^m\binom{n+m-1}{m}^{(m-1)/(2m)}
 &\le
 Cn^{-1/2}\e^{x/2}x^{B+3/4}\notag\\
 &\le C L^{B+3/4-A/2}
 =o(1),
 \label{eq:bohr-small-degree-sum}
\end{align}
where the last step is precisely \eqref{eq:A-range}.

For $m>\sqrt n$, the elementary estimate
\[
 \binom{n+m-1}{m}
 \le\left[\e\left(1+\frac nm\right)\right]^m
\]
gives
\begin{equation}\label{eq:bohr-large-degree-term}
 r^m\binom{n+m-1}{m}^{(m-1)/(2m)}
 \le r\rho_n^{m-1},
 \qquad
 \rho_n:=\left[\e r^2(1+\sqrt n)\right]^{1/2}.
\end{equation}
Since $\rho_n=O(\sqrt{\log n}\,n^{-1/4})$, one has $\rho_n\le1/2$ for
large $n$, and hence
\[
 \sum_{m>\sqrt n}m^B r\rho_n^{m-1}=O(r)=o(1).
\]
\eqref{eq:bohr-small-degree-sum} therefore yields
$\Sigma_n(r)=o(1)$.

In particular, $\Sigma_n(r)\le1/2$ for large $n$.  Since
\[
 t+\frac12(1-t^2)\le1
 \qquad(0\le t\le1),
\]
\eqref{eq:bohr-majorant} implies \eqref{eq:bohr-definition}.  Thus
$r\le\mathfrak b_n$, proving \eqref{eq:bohr-generic-lower}.
\end{proof}

\section{Proof of Theorem D: multidimensional Bohr radius}
\label{sec:proof-D}

\begin{proof}

\begingroup
\color{black}
\hypersetup{linkcolor=black,citecolor=black,urlcolor=black}
Let $B_0,K$ be as in Theorem~\ref{thm:A} and choose
\[
 A>2B_0+\frac32.
\]
\endgroup

Theorem~\ref{thm:A} and \Cref{prop:polynomial-to-bohr} give
\begin{equation}\label{eq:bohr-lower-strong}
 \mathfrak b_n
 \ge
 \sqrt{\frac{\log n-A\log\log n}{n}}
\end{equation}
for all sufficiently large $n$.

The Kahane--Salem--Zygmund construction in
\cite[Section~6]{BayartPellegrinoSeoane2014} provides, for all
$m,n\ge2$, an $m$-homogeneous polynomial
\[
 P(z)=\sum_{|\alpha|=m}c_\alpha z^\alpha
\]
such that
\begin{equation}\label{eq:KSZ-bohr}
 |c_\alpha|=\binom m\alpha,
 \qquad
 \|P\|_\infty
 \le
 \kappa\sqrt{m\log m}\,(m!)^{1/2}n^{(m+1)/2},
\end{equation}
where $\kappa$ is absolute.  Since
$\sum_{|\alpha|=m}\binom m\alpha=n^m$, the definition of $\mathfrak b_n$
gives
\[
 \mathfrak b_n^m n^m\le\|P\|_\infty.
\]
Therefore
\begin{equation}\label{eq:bohr-KSZ-upper}
 \mathfrak b_n
 \le
 \kappa^{1/m}(m\log m)^{1/(2m)}
 (m!)^{1/(2m)}n^{-1/2+1/(2m)}.
\end{equation}
Choose $m=\lfloor\log n\rfloor$.  Stirling's formula gives
\begin{align*}
 \log\!\left(
 \mathfrak b_n\sqrt{\frac n{\log n}}
 \right)
 &\le
 \frac{\log(m\log m)}{2m}
 +\frac{\log(2\pi m)}{4m}
 +\frac12\log\frac m{\log n}
 +\frac12\left(\frac{\log n}{m}-1\right)
 +O\!\left(\frac1m\right)\\
 &=
 \left(\frac34+o(1)\right)
 \frac{\log\log n}{\log n}.
\end{align*}
Hence
\begin{equation}\label{eq:bohr-upper-strong}
 \mathfrak b_n
 \le
 \sqrt{\frac{\log n}{n}}
 \exp\!\left\{
 \left(\frac34+o(1)\right)\frac{\log\log n}{\log n}
 \right\}.
\end{equation}

Combining \eqref{eq:bohr-lower-strong} and \eqref{eq:bohr-upper-strong}, and using
\[
 \sqrt{1-A\frac{\log\log n}{\log n}}
 =1-\frac A2\frac{\log\log n}{\log n}
 +O\!\left(\frac{(\log\log n)^2}{(\log n)^2}\right),
\]
proves
\[
 \mathfrak b_n=
 \sqrt{\frac{\log n}{n}}
 \left(1+O\!\left(\frac{\log\log n}{\log n}\right)\right).
\]

The same estimates, with the admissible choice $B_0=5$, yield
\begin{align}
 \limsup_{n\to\infty}
 \frac{\log n}{\log\log n}
 \left(1-\mathfrak b_n\sqrt{\frac n{\log n}}\right)
 &\le \frac{23}{4},
 \label{eq:bohr-lower-remainder}\\
 \limsup_{n\to\infty}
 \frac{\log n}{\log\log n}
 \left(\mathfrak b_n\sqrt{\frac n{\log n}}-1\right)
 &\le\frac34.
 \label{eq:bohr-upper-remainder}
\end{align}
For the first inequality, take $B_0=5$ and let
$A\downarrow23/2$. The second follows by exponentiating the upper estimate.
\end{proof}

\appendix
\begingroup
\definecolor{appendixred}{RGB}{0,0,0}
\color{black}
\hypersetup{linkcolor=black,citecolor=black,urlcolor=black}
\newcommand{\G}{\Gamma}
\section{A square-root-scale refinement}
\label{app:sqrt-scale-refinement}
\small
\setlength{\abovedisplayskip}{4pt plus 1pt minus 1pt}
\setlength{\belowdisplayskip}{4pt plus 1pt minus 1pt}
\setlength{\abovedisplayshortskip}{2pt plus 1pt}
\setlength{\belowdisplayshortskip}{2pt plus 1pt}
\setlength{\jot}{2pt}

We prove the estimate stated in Remark~\ref{rem:sqrt-scale-refinement}, using
\Cref{thm:main,sec:polynomial-bootstrap} as input.

\subsection{Weighted bootstrap and row estimates}
For $r\ge1$ put $q_r=2r/(r+1)$ and let $D_r$ be the optimal constant in
$\|a(P)\|_{q_r}\le D_r\|P\|_\infty$.  By \Cref{thm:main},
\begin{equation}\label{app-subpoly-eq-baseline-polynomial}
 D_r\le C_0r^{B_0}\qquad(r\ge1)
\end{equation}
for absolute $C_0,B_0$.  Set
\[
 \phi(t)=\sqrt t\log t,\qquad
 \phi'(t)=\frac{\log t+2}{2\sqrt t},\qquad
 \phi''(t)=-\frac{\log t}{4t^{3/2}}.
\]
For $r\ge r_0$ define
\begin{equation}\label{app-subpoly-eq-weight}
 \W(r)=\sqrt r\exp\!\bigl(\phi(\log r)\bigr),
\end{equation}
and choose the finitely many smaller values so that $\W(1)=1$, $\W$ is
increasing, and $\W(r)\ge\sqrt r$.  Let $\G_M$ be the least constant such
that
\begin{equation}\label{app-subpoly-eq-weighted-bootstrap}
 \left(\sum_{r=1}^M\frac{\|a(F_r)\|_{q_r}^2}{\W(r)^2}\right)^{1/2}
 \le \G_M\|F\|_\infty
\end{equation}
for every $F=\sum_{r=0}^M F_r$.  Then $1\le\G_M<\infty$.

\begin{theorem}\label{app-subpoly-thm-main}
There are absolute $C,m_0<\infty$ such that
\[
 D_m\le \sqrt m\exp\!\bigl(C\sqrt{\log m}\,\log\log m\bigr)
 \qquad(m\ge m_0).
\]
\end{theorem}

We use Weissler's contraction
\begin{equation}\label{app-subpoly-eq-Weissler}
 \|H(\sqrt{p/2}\,\cdot)\|_2\le\|H\|_p,
 \qquad1\le p\le2,
\end{equation}
see \cite{Weissler}.  If
$Q(x,y)=\sum a^{s,e}_{\beta\gamma}x^\beta y^\gamma$, set
\[
 X_{s,e}=\left[\sum_{|\beta|=s}
 \left(\sum_{|\gamma|=e}|a^{s,e}_{\beta\gamma}|^2\right)^{q_s/2}
 \right]^{1/q_s},\qquad
 \widehat X_{s,e}=\left(\frac{s}{s+1}\right)^{e/2}X_{s,e}.
\]
Minkowski, Parseval and \eqref{app-subpoly-eq-Weissler}, followed by
\eqref{app-subpoly-eq-weighted-bootstrap} on the $x$-sections, give
\begin{equation}\label{app-subpoly-eq-row}
 \sum_{s,e\ge1}\frac{\widehat X_{s,e}^2}{\W(s)^2}
 \le\G_M^2\|Q\|_\infty^2.
\end{equation}

For an $L$-block multidegree $\mathbf d=(d_1,\ldots,d_L)$ put
$r=\sum_jd_j$, $s_j=r-d_j$ and $\theta_j=s_j/((L-1)r)$.  Writing
$X_{j,\mathbf d}$ for the mixed norm with exponent $2$ in block $j$ and
$q_{s_j}$ in the complementary blocks, fixed-order mixed-norm log-convexity
and Minkowski yield
\begin{equation}\label{app-subpoly-eq-mixed-interpolation}
 \|a(Q_{\mathbf d})\|_{q_r}
 \le\prod_{j=1}^LX_{j,\mathbf d}^{\theta_j},
 \qquad
 \frac{\theta_i}{2}+\sum_{j\ne i}\frac{\theta_j}{q_{s_j}}
 =\frac1{q_r}.
\end{equation}
With
$\widehat X_{j,\mathbf d}=(s_j/(s_j+1))^{d_j/2}X_{j,\mathbf d}$,
combining all blocks except $j$ in \eqref{app-subpoly-eq-row} gives
\begin{equation}\label{app-subpoly-eq-multirow}
 \sum_{\mathbf d\in\mathbb N^L}
 \frac{\widehat X_{j,\mathbf d}^2}{\W(s_j)^2}
 \le\G_M^2\|Q\|_\infty^2.
\end{equation}

\subsection{Central contraction}
We shall use the elementary entropy estimate
\begin{equation}\label{app-subpoly-eq-entropy-amgm}
 \sum_\nu\prod_{j=1}^Lu_{j,\nu}^{\theta_{j,\nu}}
 \le \exp\!\left(\sup_\nu D(\theta_\nu\|\mathbf u_L)\right),
 \qquad
 D(\theta\|\mathbf u_L)=\sum_j\theta_j\log(L\theta_j),
\end{equation}
whenever $\sum_\nu u_{j,\nu}\le1$ for every $j$; it follows directly from
the log-sum inequality.

Fix $0<a<1/10$, put $\tau_L=(L-1)/L$, $\delta_L=a/L$, and
\[
 \mathcal C_r(L)=\left\{\mathbf d\in\mathbb N^L:
 \sum_jd_j=r,\ \left|d_j/r-1/L\right|\le\delta_L\ (1\le j\le L)\right\}.
\]

\begin{lemma}[Critical central contraction]\label{app-subpoly-lem-central-contraction}
There are absolute $K_0,\gamma>0$ such that, if $T=\log M$,
\begin{equation}\label{app-subpoly-eq-L-choice}
 L=\left\lceil\frac{K_0}{\phi'(T)}\right\rceil,
 \qquad R=L^4,
\end{equation}
then every $L$-block polynomial $Q$ of degree at most $M$ satisfies, for all
large $M$,
\begin{align}
 &\left[\sum_{r=R}^M\frac1{\W(r)^2}
 \left(\sum_{\mathbf d\in\mathcal C_r(L)}
 \|a(Q_{\mathbf d})\|_{q_r}^{q_r}\right)^{2/q_r}\right]^{1/2}
 \le e^{-2\gamma/L^2}\G_M\|Q\|_\infty .
 \label{app-subpoly-eq-central-contraction}
\end{align}
\end{lemma}

\begin{proof}
For $\mathbf d\in\mathcal C_r(L)$ write
$x_j=d_j/r$, $y_j=s_j/r$, $h_j=x_j-1/L$.  Thus
$\sum_jh_j=0$, $|h_j|\le a/L$, and $\theta_j=y_j/(L-1)$.  After inserting
$U_{j,\mathbf d}=\widehat X_{j,\mathbf d}/\W(s_j)$ in
\eqref{app-subpoly-eq-mixed-interpolation},
\begin{equation}\label{app-subpoly-eq-central-pointwise}
 \frac{\|a(Q_{\mathbf d})\|_{q_r}}{\W(r)}
 \le A_{\mathbf d}B_{\mathbf d}
 \prod_jU_{j,\mathbf d}^{\theta_j},
\end{equation}
where
\[
 A_{\mathbf d}=\prod_j\left(1+\frac1{s_j}\right)^{d_j\theta_j/2},
 \qquad
 B_{\mathbf d}=\prod_j\left(\frac{\W(s_j)}{\W(r)}\right)^{\theta_j}.
\]
The denormalization satisfies
$\log A_{\mathbf d}\le1/[2(L-1)]$.  Taylor expansion at
$\tau_L$ gives
\[
 \frac12\sum_j\theta_j\log y_j
 \le\frac12\log\tau_L+O(a^2/L^2),
 \qquad
 D(\theta\|\mathbf u_L)=O(a^2/L^2),
\]
and
\[
 \frac1{2(L-1)}+\frac12\log\tau_L=O(L^{-2}).
\]
For $t=\log r$ and
$g_t(y)=y[\phi(t+\log y)-\phi(t)]$, concavity of $\phi$ and
$\sum_jh_j=0$ yield
\[
 \frac1{L-1}\sum_jg_t(y_j)
 \le-\frac{\phi'(T)}L+O\!\left(\frac{a^2\phi'(t)}{L^2}\right).
\]
Using \eqref{app-subpoly-eq-entropy-amgm} and
\eqref{app-subpoly-eq-multirow}, the square sum of the individual
multidegree norms therefore has factor
\[
 \exp\!\left[-\frac{\phi'(T)}L+\frac{C_*}{L^2}\right].
\]
Choose $K_0>2C_*$.  Then \eqref{app-subpoly-eq-L-choice} leaves a factor
at most $e^{-4\gamma/L^2}$.  Finally
$\#\mathcal C_r(L)\le(r+1)^{L-1}$ and
\[
 (r+1)^{(L-1)/(2r)}
 \le \exp\!\left(\frac{C\log L}{L^3}\right)
 \le e^{\gamma/L^2}
 \qquad(r\ge L^4),
\]
which proves the claim.
\end{proof}

\subsection{Compression and closure}
Let $a_r^{\mathrm s}(F)$ contain the coefficients with
$\max_j\alpha_j>7r/8$.  By the contractive color/root recovery of
\Cref{lem:bootstrap-cap-recovery}, applied to each residual layer
$\ell<r/8$, and then interpolation with Parseval, one obtains
\begin{equation}\label{app-subpoly-eq-strong}
 \left(\sum_{r=1}^M\frac{\|a_r^{\mathrm s}(F)\|_{q_r}^2}{\W(r)^2}\right)^{1/2}
 \le C_s\G_M^{1/5}\|F\|_\infty.
\end{equation}
Indeed, with $K_r=\lfloor r/5\rfloor$, $t_r=K_r/r$, and
$b_r=\|a_r^{\mathrm s}(F)\|_2$,
\[
 \frac{\|a_r^{\mathrm s}(F)\|_{q_r}}{\W(r)}
 \le C\G_M^{1/5}r^{-1/5}b_r^{4/5},
\]
and H\"older closes the square sum.

Fix $0<\xi<1/8$ and let $a_r^{\mathrm i}(F)$ correspond to
$\xi r<\max_j\alpha_j\le7r/8$.  For a two-block polynomial, restrict the
$(d,e)$ array to rows whose largest $x$-exponent exceeds $7d/8$, and denote
the normalized row and column quantities by
$\widehat X^{\mathrm s}_{d,e}$ and $\widehat Y^{\mathrm s}_{d,e}$.  The
preceding estimate on the row side and \eqref{app-subpoly-eq-row} on the
column side give
\[
 \sum_{d,e}\frac{(\widehat X^{\mathrm s}_{d,e})^2}{\W(d)^2}
 \le C_s^2\G_M^{2/5}\|Q\|_\infty^2,
 \qquad
 \sum_{d,e}\frac{(\widehat Y^{\mathrm s}_{d,e})^2}{\W(e)^2}
 \le \G_M^2\|Q\|_\infty^2.
\]
For $d+e=r$, $t=d/r$, mixed interpolation gives the factor
$\G_M^{1-4t/5}$; moreover
\[
 \frac{\W(d)^t\W(e)^{1-t}}{\W(r)}
 \le t^{t/2}(1-t)^{(1-t)/2}\le1.
\]
Assigning each coordinate to the first block with probability $\xi/28$
recovers every intermediate parent with probability at least $\xi/56$.
Since the coefficient map is injective and
$(r+1)^{1/(2r)}\le2$, averaging gives
\begin{equation}\label{app-subpoly-eq-intermediate}
 \left(\sum_{r=1}^M\frac{\|a_r^{\mathrm i}(F)\|_{q_r}^2}{\W(r)^2}\right)^{1/2}
 \le \frac{C_i}{\xi}\G_M^{\eta}\|F\|_\infty,
 \qquad \eta=1-\frac45\xi.
\end{equation}

Let $a_r^{\mathrm d}(F)$ denote the remaining coefficients.  Color each
coordinate uniformly among $L$ blocks.  Then
\[
 \mathbb ES_j=r/L,
 \qquad
 \operatorname{Var}(S_j)\le\xi r^2/L.
\]
Bernstein's inequality \cite[Section~2.8]{BLM}, with window $ar/L$, gives
\[
 \mathbb P\{|S_j/r-1/L|>a/L\}
 \le2\exp\!\left[-\frac{a^2}{2\xi L(1+a/3)}\right].
\]
Choose
\begin{equation}\label{app-subpoly-eq-xi-choice}
 \xi=\frac{b}{L\log L}
\end{equation}
with $b>0$ small.  A union bound gives simultaneous central capture with
probability $p_0\ge1-2L^{-5}$.  No coefficients merge under the coloring,
so
\[
 p_0\|a_r^{\mathrm d}(F)\|_{q_r}^{q_r}
 \le \mathbb E_\omega
 \sum_{\mathbf d\in\mathcal C_r(L)}
 \|a((Q_\omega)_{\mathbf d})\|_{q_r}^{q_r}.
\]
The recovery factor has logarithm $O(L^{-5})$ and is absorbed in
\eqref{app-subpoly-eq-central-contraction}; hence
\begin{equation}\label{app-subpoly-eq-diffuse}
 \left(\sum_{r=R}^M\frac{\|a_r^{\mathrm d}(F)\|_{q_r}^2}{\W(r)^2}\right)^{1/2}
 \le e^{-\gamma/L^2}\G_M\|F\|_\infty.
\end{equation}

\begin{proof}[Proof of \cref{app-subpoly-thm-main}]
Set $T=\log M$, choose $L,R,\xi$ as above, normalize
$\|F\|_\infty=1$, and put $G=\G_M\ge1$.  The degrees $r<R$ contribute
\[
 A_M=\left(\sum_{r<R}\frac{D_r^2}{\W(r)^2}\right)^{1/2}
 \le CL^{4B_0},
 \qquad \log(1+A_M)\le C\log L.
\]
Using \eqref{app-subpoly-eq-strong},
\eqref{app-subpoly-eq-intermediate}, and
\eqref{app-subpoly-eq-diffuse},
\[
 G\le A_M+e^{-\gamma/L^2}G+C_sG^{1/5}
       +\frac{C_i}{\xi}G^\eta,
 \qquad \eta=1-\frac45\xi.
\]
Since $G\ge1$ and $\eta>1/5$ for large $L$,
\[
 (1-e^{-\gamma/L^2})G
 \le\left(A_M+C+\frac C\xi\right)G^\eta.
\]
Thus
\[
 \log G
 \le \frac1{1-\eta}\left[
 \log\!\left(A_M+C+\frac C\xi\right)
 +\log\frac1{1-e^{-\gamma/L^2}}\right]
 \le C L(\log L)^2.
\]
Since
$L\asymp\sqrt T/\log T$, this gives
\begin{equation}\label{app-subpoly-eq-Gamma-target}
 \log\G_M\le C\sqrt{\log M}\,\log\log M.
\end{equation}
Finally \eqref{app-subpoly-eq-weighted-bootstrap} applied to one
$M$-homogeneous polynomial, together with \eqref{app-subpoly-eq-weight},
yields
\[
 D_M\le\G_M\W(M)
 \le\sqrt M\exp\!\bigl(C\sqrt{\log M}\,\log\log M\bigr).
\]
\end{proof}

\begin{remark}
The loss comes from the $L^{-2}$ diffuse contraction and the choice
$\xi\asymp(L\log L)^{-1}$ forced by simultaneous Bernstein capture.
\end{remark}

\endgroup

\section*{Acknowledgments}

\paragraph{\bf Fundings.} D. Pellegrino is supported by Grants No.~406457/2023-9
(CNPq/MCTI N\textsuperscript{o}~10/2023), No.~403964/2024-5
(MCTI/CNPq N\textsuperscript{o}~16/2024), and No.~305807/2025-0 from the
Conselho Nacional de Desenvolvimento Cient\'ifico e Tecnol\'ogico (CNPq,
Brazil). E. Teixeira gratefully acknowledges support from the Grayce B. Kerr
Chair funds at Oklahoma State University.

This work was conducted in part within the DARPA ExpMath project
\emph{``A Human-Centered Framework for AI-Mathematician Collaboration in
Research-Level Mathematics''} (Agreement No.~HR0011262E029), in which
E. Teixeira serves as a co-principal investigator and gratefully acknowledges
partial support.  The views and conclusions expressed here are those of the
authors and should not be interpreted as representing the official policies
of the Department of Defense or the U.S.\ Government.

\bigskip

\paragraph{\bf Computational verification and machine assistance disclosure.}

The authors conceived the research program, formulated its central ideas, and directed the successive development of its statements and arguments. During the preparation of the paper,
the \textsc{Lea} prover, developed within the ExpMath project, and ChatGPT 5.6 Sol (OpenAI) were used as a collaborative research tool to stress-test ideas and support exploratory analysis, including the examination of parameter ranges and the refinement of estimates. It also assisted with consistency checks, the organization of arguments, literature searches, and routine \LaTeX{} typesetting. This assistance included comparing formulations of the coloring and compression steps, testing parameter ranges for diffuse capture and dominant summability, and
the elementary contraction estimate in the proof of
\Cref{lem:elementary-admissible}. {\definecolor{disclosurepurple}{RGB}{0,0,0}\color{black}\hypersetup{linkcolor=black}This assistance also
covered the square-root-scale refinement in Appendix~\ref{app:sqrt-scale-refinement}.}
The authors retain intellectual authorship of the work. They wrote and revised the manuscript with the assistance described above and independently verified all mathematical statements and proofs. They reviewed the references and take full responsibility for the correctness, originality, and integrity of the work.

The mathematical results and proofs of this paper have been formalized
in Lean~4 using \textsc{Lea}, the formalization system developed within
our DARPA ExpMath project. The formalization is available at
\begin{center}
\url{https://github.com/teixeiramath-maker/PolynomialGrowth-BH-ComplexPol}.
\end{center}

%

\end{document}